\documentclass[10 pt]{amsart}

\usepackage{latexsym, amsmath, amssymb, longtable, booktabs,amscd,microtype,booktabs,cases}
\usepackage[square]{natbib}
\usepackage{graphicx}
\usepackage[dvipsnames]{xcolor}
\usepackage{caption}
\usepackage{tcolorbox}
\usepackage{tikz}
\usepackage{tikz-cd}
\usepackage[all]{xy}
\usepackage{lipsum}

\usepackage{hyperref}
\hypersetup{
    colorlinks,
    citecolor=blue,
    filecolor=black,
    linkcolor=red,
    urlcolor=black
}
\bibpunct{[}{]}{,}{n}{}{;} 

\numberwithin{equation}{section}
\newcommand\Frob{{\mathrm{Frob}}}
\newcommand\Br{{\mathrm{Br}}}

\newcommand{\Z}{\mathbb{Z}}
\newcommand{\F}{\mathbb{F}}

\newcommand{\Ga}{\Gamma}
\newcommand{\ga}{\gamma}
\newcommand{\de}{\delta}

\newcommand{\C}{\mathbb{C}}

\newcommand\A{\mathbb{A}}

\newcommand\N{\mathbb{N}}

\newcommand\sO{\mathcal{O}}

\newcommand\GL{{\mathrm{GL}}}

\newcommand\HH{{\mathrm H}}

\newcommand{\Q}{\mathbb{Q}}
\newcommand\Gal{{\mathrm{Gal}}}
\newcommand\Inf{{\mathrm{Inf}}}
\newcommand\Ind{{\mathrm{Ind}}}
\newcommand\Aut{{\mathrm{Aut}}}

\newtheorem{thm}{Theorem}[section]
\newtheorem{theorem}[thm]{Theorem}
\newtheorem{cor}[thm]{Corollary}

\newtheorem{prop}[thm]{Proposition}
\newtheorem{lemma}[thm]{Lemma}
\theoremstyle{definition}
\newtheorem{definition}[thm]{Definition}
\newtheorem{remark}[thm]{Remark}

\theoremstyle{definition}

\theoremstyle{remark}

\theoremstyle{remark}

\makeatletter
\def\imod#1{\allowbreak\mkern10mu({\operator@font mod}\,\,#1)}
\makeatother

	\title{Supercuspidal ramifications and traces of adjoint lifts}
	\author{Debargha Banerjee}
	\author{Tathagata Mandal}
	\address{INDIAN INSTITUTE OF SCIENCE EDUCATION AND RESEARCH, PUNE, INDIA}
\begin{document}

\begin{abstract}
   We write down 
    the local Brauer classes of the endomorphism
	algebras of motives attached to  non-CM Hecke eigenforms for {\bf all}
	supercuspidal primes in terms of traces of adjoint lifts at auxiliary primes. 
	We give an alternative proof of the ramification formulae for odd primes 
	obtained by Bhattacharya-Ghate
	and write down the ramification formulae for odd unramified 
	supercuspidal primes of level zero also removing a mild hypothesis of
	them.
	We also give a complete description of ramifications for $p=2$
	where the local Galois representation can be non-dihedral. In the process, we write 
	down the inertial Galois representations even for $p=2$ also generalizing similar description 
	of the same for odd primes by Ghate-M\'ezard.  
    Some numerical examples using {\texttt Sage} and {\texttt LMFDB} are provided
    supporting some of our theorems.   
\end{abstract}
	
	\maketitle

 \section{Introduction}
  Let $f(z)=\sum_{n \geq 1} a_n q^n \in S_k(N, \epsilon) $ be a non-CM 
  Hecke eigenform  of weight $k \geq 2$, level $N \geq 1$ and nebentypus $\epsilon$.
  One knows that the number field $E=\Q(\{a_n\})$ 
  is either a totally real or a CM number field. 
  For $k=2$, let $M_f$ denote the abelian variety attached to $f$ 
  by Shimura \cite{MR0314766}. 
  For $k >2$,  we also denote by $M_f$ the Grothendieck motive over $\Q$ with coefficients in $E$ associated to $f$ by Scholl \cite{MR1047142}. 
  The $\lambda$-adic realization of this motive produces 
  a $\lambda$-adic Galois representation $\rho_f$ associated to the modular form $f$ by a well-known theorem of Deligne [cf.  Section~\ref{galrep}].  
  Let $X_f$ denote the $\Q$-algebra of endomorphisms of $M_f$ defined by 
  $X=X_f:= \text{End}_{\overline{\Q}}(M_f) {\otimes}_{\Z} \Q$.

  Consider the totally real subfield $F$ of $E$ generated by the elements 
  $a_p^2 \epsilon(p)^{-1}$, for all $p \nmid N$. 
  The algebra $X$ has the structure of an explicit crossed product algebra over $F$
  due to Momose and Ribet \cite{MR594532} 
  in weight two and Ghate and his collaborators 
  \cite{MR2038777}, \cite{MR2146605}, \cite{MR1650241} in higher weights.
  One knows that the class $[X] \in
  {}_2\Br(F):={}_2\mathrm{H}^2(\Gal(\bar{F}|F), \bar{F}^\times)$, 
  the $2$-torsion part of the Brauer group of $F$.
  In \cite{MR594532}, Ribet wondered if it is possible to determine 
  the local Brauer classes by pure thought. 
  We study the Brauer classes of $X$ locally by the following 
  exact sequence:
  $$0 \to {}_2{\Br(F)} \to \oplus_v 
  \,\,{}_2{\Br(F_v)} \to \Z/2 \to 0 $$
  where $v$ runs over all primes in $F$. It is well-known that $X_v=X \otimes_F F_v$
  is a central simple algebra over $F_v$ and the class of $X_v$ is a $2$-torsion
  element in the Brauer group $\Br(F_v)$ of $F_v$, that is, the class 
  $[X_v] \in {}_2\Br(F_v) \cong \Z/2$. 
  We say $X_v$ is unramified if the class of $X_v$ is trivial
  and ramified if the class is non-trivial.
  At the infinite places $v$, $X$ is totally indefinite if $k$ is even,
  and totally definite if $k$ is odd \cite[Theorem $3.1$]{MR617867}.
  The ramification formula of $X_v$ is known
  by a series of papers pioneered by Ghate and his collaborators
  \cite{MR2770587}, \cite{MR3096563}, \cite{MR2038777} and \cite{MR2146605}
  for non-supercuspidal primes $p$
  (i.e., the local automorphic factor at $p$ is not of
  supercuspidal type).
  For $v \mid p$, let $G_p:=\Gal(\bar{\Q}_p|\Q_p)$
  and $G_v:=\Gal(\bar{F}_v|F_v)$ be the local Galois groups.
  \begin{definition}
  	We call a supercuspidal prime $p$ to be  dihedral for 
  	$f$ if  the local Galois representation
  	$\rho_f|_{G_p} \sim \Ind_{G_K}^{G_p}\chi$
  	for some quadratic extension $K|\Q_p$ and some character $\chi$
  	of $G_K:=\Gal(\bar{\Q}_p|K)$. 
  	Depending on $K|\Q_p$ is unramified
  	(or ramified), we call  the prime $p$ to be a unramified (or ramified)
  	supercuspidal prime for $f$.  
  	By the level of 
  	an unramified supercuspidal prime $p$, 
  	we mean the level of the corresponding 
  	local automorphic representation $\pi_p$.
  \end{definition}
  
  We say $f$ is $p$-minimal, if the $p$-part of its level 
  is the smallest among all twists
  $f \otimes \psi$ of $f$ by Dirichlet characters $\psi$.
  Write $N=p^{N_p}N'$ with $p \nmid N'$ and 
  the nebentypus $\epsilon = \epsilon_p \cdot \epsilon'$
  as a product of its $p$-part and prime to $p$-part with
  cond$(\epsilon_p)=p^{C_p}$. The supercuspidal primes $p$ for $f$
  can be characterized as follows:
  $C_p < N_p \geq 2$ with $a_p=0$. 
  If $p$ is a supercuspidal prime, then $a_p=0$ and 
  the corresponding slope is infinity and it is not possible to talk about 
  the parity of slopes.
  In \cite{MR3391026}, Bhattacharya-Ghate determined the Brauer classes of $X_v$ for
  odd supercuspidal primes (where the local Galois representation is always dihedral)
  with an extra hypothesis in the level $0$ case
  using a method involving local symbols and $(p,\ell)$-Galois representation. 
  They determined $X_v$ except in the following cases:
  \begin{enumerate}
  	\item 
  	$p=2$,
  	\item
  	$p$ is odd, $K=\Q_p(s)$ is unramified, the $p$-minimal twist of $f$
  	is of level $0$ and $\chi(s)^{p-1}=-1$.
  \end{enumerate}
  
  We wish to determine the algebra $X_v$ for \textbf{all} supercuspidal primes
  ({\it including} $p=2$ where the local Galois representation can be non-dihedral
  and the case $(2)$ above) in terms of the \emph{companion adjoint slope} $m_v$ 
  [cf. Definition~\ref{companion}] determined by the traces of adjoint lifts at the auxillary primes. 
 Note that the Theorem~\ref{nonexceptionalodd} and a part of Theorem~\ref{extrathm} 
  proved in this article are just a restatement of the results in \cite{MR3391026} for odd primes. However, our method is completely different as it uses local nature of Galois representations associated to modular forms. 
   
   In \cite{MR3391026}, the authors use a formula that determines the local Brauer classes of $X_v$
  as a product of finite number of local Hilbert symbols 
  \cite[Theorem $4.1$]{MR2146605}. In loc. cit., the authors simplify
  this product and give the ramification formula of $X_v$
  for odd supercuspidal primes.
  As the formula of local symbols for $v \mid 2$ (except $F=\Q$)
  is a bit technical \cite[Chapter $7$, p. $241$]{MR1218392},
  it is not possible to simplify the product in a similar way. 
  Our method circumvents the problem as we are using the local information 
  about modular forms determined by the corresponding filtered $(\phi,N)$-modules
  associated to $(p,p)$-Galois representation. Group cohomology plays a key role in our computation. 
  In Section~\ref{chi}, we write down the inertial Galois representations even for $p=2$. It will be of independent interest similar to that for $p$ odd.

  For supercuspidal primes (specially for $p=2$), the 
  image of the inertial groups under modular Galois representations are complicated and hence our results are also a bit theoretical in nature.  Our results are technical compare to Bhattacharya-Ghate because they used crucially  [cf. \cite[Lemma 4.2]{MR3391026} for instance] the order 
  of the Brauer class is co-prime to the residue degree of the local place. The ingenuity of our results stems from the fact that 
  we manage to identify the cases when the local Brauer classes are {\it not} determined by adjoint 
  slopes. Note that for most of the places of $\Q$, they are related to adjoint slopes. 
  In an ongoing project, we wish to find out the  local Brauer class
  of the endomorphism algebra of motive attached to a Hilbert modular form \cite[Appendix B]{MR2962318} using the computation done in the present paper. 
  
  \noindent
  \textbf{Notation.}
  Throughout the paper, $I_p$ and $I_v$ be the inertia 
  subgroups of $G_p$ and $G_v$ respectively.
  The wild and the tame part of the inertia subgroup $I_K$ 
  of $G_K$ will be denoted by $I_W(K)$ and $I_T(K)$ respectively.
 \subsection{Acknowledgement}
 The first named author was partially supported by the SERB grant YSS/2015/ \\
 001491. The authors were greatly benefitted from  several e-mail communications,
 encouragement of Professor Eknath Ghate. It is also a great pleasure 
 to acknowledge several help of Dr. Shalini Bhattacharya 
 during the preparation of this article.
	
 \section{The local Brauer Class of X and Its invariant}
 \label{Brauerclasscocycle}
 Let $
 \Ga= \{\ga \in \Aut(E): \,\,\exists \,\, \text{a Dirichlet character}\,\,
 \chi_\ga \,\, \text{such that} \,\, a_p^\ga=a_p \cdot \chi_\ga(p) \,\,
 \text{with} \,\, (p,N)=1 \}
 $
 be the group of extra twists for $f$. One knows that
 $F$ is the fixed field of $E$ by $\Ga$. 
 For $\gamma, \delta \in \Gamma$, the relation 
 $\chi_{\gamma \delta}=\chi_\gamma \chi_\delta^\gamma$ shows that
 $\gamma \mapsto \chi_\gamma(g)$ is a $1$-cocycle for a fixed $g \in G_\Q$. 
  Since $H^1(\Gamma, E^\times)$ is trivial, 
 there exists $\alpha(g) \in E^\times$ such that
 \begin{eqnarray}\label{1}
	 \alpha(g)^{\gamma -1} = \chi_\gamma (g),
 \end{eqnarray}
 for all $\gamma \in \Gamma$ by Hilbert's theorem $90$. 
 The element $\alpha(g)$ is well defined
 modulo $F^\times$. The map $\tilde{\alpha}: G_\Q \to E^\times / F^\times,
 \,\, g \mapsto \alpha(g) \,\, \text{mod} \,\, F^\times$ is a continuous 
 homomorphism. The map $\alpha: G_\Q \to E^\times$ can be thought of as a lift of 
 $\tilde{\alpha}$. 
 Let $\rho_f$ denote the $\lambda$-adic representation attached to $f$ 
 for some prime $\lambda \, | \, \ell$ of $E$. We list some properties of any lift $\alpha$ of the
 homomorphism $\tilde{\alpha}$.

\begin{prop} \label{prop1} 
 \cite[Lemma $1$]{MR2770587} \cite[Theorem $5.5$]{MR2058653}
	The map $\alpha$ satisfies the following properties:
	\begin{enumerate}
		\item 
			 $\alpha^2(g) \equiv \epsilon(g) \,\, \text{mod} 
			 \,\, F^\times$, for all $g \in G_\Q$.
		\item
			 $\alpha(g) \equiv \mathrm{Tr}(\rho_f(g)) \,\, \text{mod} \,\, F^\times$,
			 for all $g \in G_\Q$, provided that the trace is non-zero.
		\item
			 $\alpha(\Frob_p) \equiv a_p \,\, \text{mod} \,\, F^\times$, 
			 for all prime $p \nmid N$ with $a_p \neq 0$.
	\end{enumerate}
\end{prop}
    Comparing (\ref{1}) and the property ($1$) of the above proposition, we have
    the following identity:
    $
    \chi_\ga^2=\epsilon^{\ga -1}
    $,	
    for all $\ga \in \Ga$.
 	According to \cite{MR633903}, the Brauer class of $X$ in 
	$\Br(F)={\HH^2}(G_F, \bar{F}^\times)$ is given 
	by the $F^\times$-valued $2$-cocycle:
	$
	  c_\alpha(g,h)=\frac{\alpha(g) \alpha(h)}{\alpha(gh)}
	  \,\, \forall \,\, g,h \in G_F,
	$
	for any continuous lift $\alpha$ of $\tilde{\alpha}$ and this class is 
	independent of the lift choosen \cite[Section $3.2$]{MR3096563}. The restriction 
	$[c_\alpha|_{G_v}] \in \text{Br}(F_v) ={\HH}^2(G_v, \bar{F_v}^\times)$
	gives the local Brauer class of $X_v$ for any prime $v$ of $F$. 
	Since $\mathrm{inv}_v(c_\alpha|_{G_v}) \in \frac{1}{2}
	\Z/\Z$, the invariant map $\mathrm{inv}_v$ at $v$ 
	completely determines the class $[X_v]$ in $\text{Br}(F_v)$.
	For the definition of the invariant map,
	we refer to \cite[Ch. XIII, Section $3$, p. $193$]{MR554237}. 
	Let $S: G_v \to \bar{F_v}^\times$ be any set map. 
	Recall the following lemma \cite[Lemma $9$]{MR3096563} useful to
	determine  the Brauer class of any local $2$-cocycle  of the form:
	$
	  c_S(g,h)=\frac{S(g)S(h)}{S(gh)};
	$ for all $g,h \in G_v$.
	
	\begin{lemma} \label{lem1}
		Let $S:G_v \to \bar{F_v}^\times$ be any map and 
	    $t: G_v \to \bar{F_v}^\times$ be an unramified
		homomorphism with
		\begin{enumerate}
			\item 
			$S(i) \in F_v^\times$, for all $i \in I_v$,
			\item
			$S(g)^2/t(g) \in F_v^\times$, for all $g \in G_v$.
		\end{enumerate}
		For any arithmetic Frobenius $\Frob_v$, we then have 
		$
		  \mathrm{inv}_v(c_S)= \frac{1}{2} 
		  v \Big(\frac{S^2}{t}(\Frob_v) \Big)
		  \,\, \text{mod} \,\, \Z \in \frac{1}{2}\Z/\Z.
		$
		Here, $v: F_v^\times \to \Z$ is the surjective valuation.
	\end{lemma}

  For $\gamma \in \Gamma$, there is a unique Dirichlet character $\chi_\gamma$
  such that $f^\gamma \equiv f \otimes \chi_\gamma.$
  By restricting to the corresponding decomposition group $G_p$, we deduce that the $p$-adic Galois representations are similar. In other words, $\rho_{f^\gamma,p} \sim \rho_{f,p} \otimes \chi_\gamma$.
  Using Brauer-Nesbitt theorem, we have $\rho_{f^\ga,p} \sim \rho_{f,p}^\ga$.
  Hence, the property ($2$) of Proposition~\ref{prop1} is also true even for 
  $p$-adic Galois representation by comparing the traces of the similar $p$-adic Galois 
  representations associated to $f^\gamma$ and $f \otimes \chi_\gamma$.

  Since $c_\alpha|_{G_v}$ determines the algebra $X_v$, the most obvious choice
  for $S$ in the above lemma would be $\alpha$.
  The main difference of the computation of $X_v$ for $p$ odd and $p=2$
  is as follows:
  for odd $p$, we will see that $S=\alpha$ except $KF_v|F_v$ is ramified quadratic
  or $p$ is a bad prime [cf. Definition~\ref{goodbad}]. When $S \neq \alpha$,
  we have to divide $\alpha$ by a suitable auxiliary function (i.e., $S=\alpha/f$)
  to make the above lemma applicable. For $p=2$, we will always be needed to divide
  $\alpha$ by (one or more) auxiliary functions unless $N_2=2$ and their
  corresponding cocycles will contribute in the ramification of $X_v$
  as error terms.
  
\section{Statement of results}
\label{Statement} 
 If the Brauer class of $X_v$ is determined by the parity of an integer
 $a$, we write $X_v \sim a$ or $X_v \sim (-1)^a$.
 Choose a prime $p'$ coprime to $N$, with non-zero 
 Fourier coefficients $a_{p'}$, satisfying the following properties:
 \begin{equation} \label{cong1}
 p' \equiv 1 \pmod {p^{N_p}}, \quad p' \equiv p \pmod {N'}.
 \end{equation}
 \begin{definition}\label{companion}
 	Let $v$ be a valuation on $F$ such that $v(p)=1$, 
 	We define the ``\emph{companion  adjoint slope}" at a place $v$ of $F$  lying above  a supercuspidal prime $p$ to be the $v$-adic valuation of the trace  of adjoint lift at $p'$. In other words, 
 	$
 	m_v:=[F_v:\Q_p]  \cdot v(a_{p'}^2\epsilon(p')^{-1})
 	$
 	denote the ``companion adjoint slope" at $v$.
 \end{definition}
 We choose the following auxiliary primes with non-zero Fourier coefficients:
 \begin{itemize}
 	\item
 	$p^{''} \equiv 1  \pmod{N'}$ and $p^{''}$ has order $(p-1)$ in 
 	$(\Z/p^{N_p}\Z)^\times$, 
 	\item 
 	$p^{'''} \equiv 1 \pmod{N'}$ and $p^{'''}$ has order $2$ in
 	$(\Z/2^{N_2}\Z)^\times$, 
 	\item
 	for all $\gamma \in \Gamma$,
 	\begin{eqnarray*}
 		\chi_\gamma(p^\dagger) =
 		\begin{cases}
 			-1, & \quad \text{if} \,\, 
 			\chi_\gamma \,\, \text{is ramified},\\
 			1, & \quad \text{if} \,\, \chi_\gamma \,\, 
 			\text{is unramified}. \\
 		\end{cases}
 	\end{eqnarray*}
 \end{itemize}
 There exist infinitely many such primes since $f$ is assumed to be non-CM. 
 For an odd supercuspidal prime $p$, the local Galois representation is
 always dihedral, i.e., 
 $
 \rho_f|_{G_p} \sim \Ind_{G_K}^{G_p} \chi,
 $
 with $K|\Q_p$ quadratic.
 According to \cite{MR2471916},  we have
 $\chi|_{I_p} = \omega_2^l \cdot \chi_1 \cdot \chi_2$ (when $K$ is unramified)
 and $\chi|_{I_K} = \omega^l \cdot \chi_1 \cdot \chi_2$ (when $K$ is ramified)
 [cf. Section~\ref{chi}].
 Here, $\omega_2$ is the fundamental character of level $2$, 
 $\omega$ is the Teichm\"{u}ller character and $\chi_m$ is the character
 having some $p$-power order for $m=1,2$.
 
 \begin{definition} \label{goodbad}
 	We call an odd unramified supercuspidal prime $p$ of level zero  
 	to be ``good" if 
 	\begin{center}
 	    \textbf{(H)} $l$ is not an odd multiple of $(p+1)/2$.
    \end{center}
 	If such a prime $p$ is not 
 	``good", we call it a ``bad" level zero unramified supercuspidal prime. 
 \end{definition} 
 In Lemma~\ref{lem12}, we prove that level zero 
 unramified supercuspdial primes $p \equiv 1 \pmod{4}$ with $C_p=0$
 and $p \equiv 3 \pmod{4}$ with $C_p=1$ are always ``good". 
 For $a,b \in F_v^\times$,
 we write $a=\pi_v^{v(a)} \cdot a'$ and $b=\pi_v^{v(b)} \cdot b'$,
 where $\pi_v$ is a uniformizer in $F_v$. The corresponding local symbol 
 \begin{equation} \label{localsymbol}
 (a,b)_v=(-1)^{v(a)v(b) \frac{Nv-1}{2}} \cdot \big(\frac{b'}{v}\big)^{v(a)}
 \cdot \big(\frac{a'}{v}\big)^{v(b)}.
 \end{equation}
 Here $\big(\frac{.}{v}\big)$ is the local quadratic residue symbol
 in the residue field at $v$.
 
 We now fix a uniformizer $\pi$ in $K$ and let $g_\pi \in G_K$ be an element which is
 mapped to $\pi \in K^\times$ under the reciprocity map.
 Note that $g_\pi$ is a Frobenius element in $G_K$.
 Assume that 
 \[
   \alpha(g_\pi) \equiv b \,\, \text{mod} \,\, F_v^\times. 
 \] 
 Observe that the element $b$ depends on the choice of the Frobenius.
 Any other Frobenius in $G_K$ has the form $g_\pi i$ for some $i \in I_K$, and 
 $\alpha(g_\pi i) \equiv \alpha(g_\pi) \alpha(i) \equiv
 b \cdot \alpha(i) \pmod {F_v^\times}$.
 The dependence of the result on the choice of uniformizer is not surprising as \cite[Theorem~$6.1$]{MR3391026} also depends on a fixed uniformizer 
 $s$ of the quadratic unramified extension $K=\Q_p(s)$ of $\Q_p$.

Consider the field $F_v^{'}=F_v(b)$ and  let us now define the following error terms:
 \begin{equation}\label{c}
 (-1)^{n_v} = 
 \begin{cases}
 (\pi^2,a_{p^{''}}^2)_v, \quad \text{if $p$ is an odd ramified
 	supercuspidal prime  with $KF_v|F_v$ ramified quadratic}, \\
 (t,c)_v, \quad \text{if $p$ is an odd unramified  ``bad" supercuspidal  prime}, \\
 \end{cases}
 \end{equation}
 where $\pi$ is a uniformizer in $K$,
 $c \in F_v^\times$ is given by $\alpha(i) \equiv \sqrt{c}$ mod $F_v^\times \,\,
 \forall \,\, i \in I_T(F_v)$ with $\alpha(i) \notin F_v^\times$ 
 [cf. Equ. \ref{definitionofc}] and
 $t \in F_v^\times$ is the quantity given by the quadratic extension $F_v(\sqrt{t})|F_v$
 cut out by the quadratic character $\psi$ defined as follows:
 $\psi(g)=1, \,\, \text{if} \,\, \alpha(g) \in (F_v^{'})^\times \,\, \text{and} \,\,
 \psi(g)=-1, \,\, \text{if} \,\, \alpha(g) \notin (F_v^{'})^\times$.

 Note that $KF_v=F_v$ if and only if $K \subseteq F_v$.
 When $K \nsubseteq F_v$, the extension $KF_v|F_v$ turns out to be an unramified quadratic extension in the following cases \cite{MR3391026}:
 $p$ is an odd unramified supercuspidal prime or
 $p \equiv 1 \pmod{4}$ is a ramified supercuspidal prime or
 $p \equiv 3 \pmod{4}$ ramified supercuspidal prime with
 the ramification index $e(F_v|\Q_p)=e_v$ even.
 In the remaining case, that is, when $p \equiv 3 \pmod{4}$ is a
 ramified supercuspidal prime with $e_v$ odd,
 the extension $KF_v|F_v$ becomes ramified quadratic.

 \begin{theorem}
 	\label{nonexceptionalodd}
 	Let  $v$ be a place of $F$
 	lying above an odd supercuspidal prime $p$ for $f$ satisfying one of the following properties:
 	\begin{enumerate}
 		\item
 		$p$ is an unramified supercuspidal prime 
 		of positive level 
 		or it is  a ``good" level zero 
 		unramified supercuspidal prime,
 		\item
 		$p$ is a ramified supercuspidal prime with $K \subseteq F_v$ or $KF_v|F_v$ unramified quadratic extension.
 		
 	\end{enumerate}
 	The local endomorphism algebra $X_v$ is a matrix algebra if and only if $m_v$ is even.
 \end{theorem}

 We wish to emphasis that the above result 
 is exactly the same as \cite{MR3391026}.
 For a ``good" level zero 
 unramified supercuspidal prime,
 the hypothesis \textbf{(H)} here is exactly the same as the condition 
 of \cite[Theorem $6.1$]{MR3391026}
 [cf. Lemma~\ref{lem11}].
 Observe that the hypothesis of \cite[Theorem $6.1$]{MR3391026}
 is not required for level zero unramified supercuspidal primes
 $p \equiv 1 \pmod{4}$ with $C_p=0$
 and $p \equiv 3 \pmod{4}$ with $C_p=1$ [cf. Lemma~\ref{lem12}]. 
 In the case of odd unramified supercuspidal primes
 for $f$ of level zero without the hypothesis,  we predict the ramifications of endomorphism algebras using the following theorem:
 \begin{theorem}  
 	\label{extrathm}
 	Let $v \mid p$ be a place of $F$ with 
 	$p$  a ``bad" level zero unramified supercuspidal prime 
 	or
 	$KF_v|F_v$ is a ramified quadratic extension. 
 	The 
 	ramification of 
 	$X_v$  is determined by the parity of $m_v+n_v$. 
 \end{theorem}
 \begin{remark}
 The result obtained in the present article  for $KF_v|F_v$ ramified quadratic extension  
 is  same as that of \cite{MR3391026} by Lemma~\ref{equiv}.
 \end{remark}
 We now consider the case $p=2$.  Let $\omega_2$ be the fundamental character of level $2$, $\omega$
 is a trivial character and $\chi_m$ is the character on a cyclic group
 of some $2$-power order generated by $\ga_m$ for $m=1,2$.
 In the dihedral supercuspidal case,
 we show that the inertia type can be written as follows [cf. Section~\ref{chi}]:
 $\chi|_{I_2} = \omega_2^l \cdot \chi_1 \cdot \chi_2$ (when $K$ is unramified)
 and $\chi|_{I_K} = \omega \cdot \chi_1 \cdot \chi_2$ (when $K$ is ramified). 
 Assume that $\chi_1$ takes $\gamma_1$ to $\zeta_{2^r}$ 
 and $\chi_2$ takes $\gamma_2$ to $\zeta_{2^s}$.

To define the error terms,  let $b$ be as above and consider two fields 
 $F_v^{'}=F_v(b, \zeta_{2^s}+\zeta_{2^s}^{-1})$
 and $F_v^{''}=F_v(b, \zeta_{2^r})$. 
  We now define two characters $\psi_1, \psi_2$ on $G_v$ as follows:
 \begin{gather*} 
 \psi_1(g)=
 \left\{\begin{array}{lll}
 1 & \quad \text{if} \,\, \alpha(g) \in (F_v^{'})^\times \\              
 -1 & \quad \text{if} \,\, \alpha(g) \notin (F_v^{'})^\times
 \end{array}\right.
 \quad \text{and}  \quad	\psi_2(g)=
 \left\{\begin{array}{lll}
 1 & \quad 
 \text{if} \,\, \alpha(g) \in (F_v^{''})^\times \\             
 -1 & \quad 
 \text{if} \,\, \alpha(g) \notin (F_v^{''})^\times.
 \end{array}\right.
 \end{gather*}

Denote by $F_v(\sqrt{t_1}), F_v(\sqrt{t_2})$, the quadratic extensions of $F_v$ cut out by the characters $\psi_1$ and  $\psi_2$. 
 Define an integer $n_v$ modulo $2$ as follows:
 \begin{equation*}
 (-1)^{n_v} =
 \begin{cases}
 (t_1, \zeta_{2^{r-1}})_v \cdot  (t_2,(\zeta_{2^s} + 
 \zeta_{2^s}^{-1})^2)_v, \quad  \text{if $p=2$ and  $s \neq 2 $,} \\
 (t_2, a_{p^\dagger}^2)_v, \quad  \text{if $p=2$ and $s=2$.} \\
 \end{cases}
 \end{equation*}
 Consider an element
 $d_0 \in F_v^\times$  given by
 $\alpha(i) \equiv \sqrt{d_0}$ mod $F_v^\times \,\, \forall \,\, 
 i \in I_T(F_v) \setminus I_T(KF_v)$. An easy check using Lemma~\ref{lem8} shows 
 that $d_0$ is well-defined. 
We also define two integers $n_v^{'}, n_v^{''}$
 mod $2$ by
 $
 (-1)^{n_v^{'}}= (t_1, \zeta_{2^{r-1}})_v \cdot  
 (t_2,(\zeta_{2^s}+\zeta_{2^s}^{-1})^2)_v \cdot
 (\pi^2, d_0)_v
 $
 and
 $(-1)^{n_v^{''}}= (t_2, a_{p^\dagger}^2)_v \cdot (\pi^2, d_0)_v$, 
  Note that these  error terms can be explicitly computed 
 from the information about a given modular form following \cite{MR1218392}.

 We now state our main theorem for {\it dihedral} supercuspidal prime $p=2$. 
 \begin{theorem}
 	\label{nonexceptional2}
 	Let $p=2$ be a dihedral  supercuspidal prime for $f$ and 
 	$v$ be a place of $F$ lying above  prime $p$. The ramification of $X_v$ is determined by the parity of $m_v+r_v$. 
 	\begin{enumerate}
 		\item
 		If $K \subseteq F_v$ or  $KF_v|F_v$ is an unramified quadratic extension, 
 		then the error term is $r_v=n_v$.	
 		\item
 		Assume $KF_v|F_v$ is a ramified quadratic extension.
 		\begin{itemize}
 			\item 
 			If $\zeta_{2^s}+\zeta_{2^s}^{-1} \neq 0$, 
 			then the error term 
 			$r_v=n_v^{'}$. 
 			\item 
 			For $\zeta_{2^s}+\zeta_{2^s}^{-1}=0$
 			the error term  is given by
 			$r_v=n_v^{''}$.
 		\end{itemize}
 	\end{enumerate}
 \end{theorem}
 
 In case ($2$) of the above theorem
 with $F=\Q$, we will prove that the quantity $d_0$ 
 in the error term is equal to $a_{p'''}^2$
 except $K=\Q_2(\sqrt{d})$ with $d=2,-6$.
 The following corollary determines the situation of the above 
 theorem when the local algebra $X_v$ is determined by the
 parity of $m_v$ itself. 
 \begin{cor}
 	\label{maincoro}
 	Let $p=2$ be a dihedral supercuspidal prime for $f$ with $N_2=2$.
 	The ramification of the local Brauer class
 	of $X_v$ is determined by the parity of $m_v$, for any $v \mid 2$.
 \end{cor}
 Let $\rho_2(f)$ be the local representation of the Weil-Deligne group of $\Q_2$ 
 associated to $f$ at the prime $p=2$.
 When inertia acts irreducibly, the projective image of $\rho_2(f)$
 is isomorphic to one of three ``exceptional" groups $A_4,S_4,A_5$.
 For any $v \mid 2$,
 let $D_{K'}$ be the discriminant of the field $K'$ cut out by the kernel of the 
 homomorphism $d:G_v \to F_v^\times/F_v^{\times 2}$ with
 $d=\frac{\alpha^2}{D}$ and $D=\det(\rho_2(f))$.
 In this case, we prove:
 
 \begin{theorem} \label{exceptional}
 	Let $p=2$ be a non-dihedral supercuspidal prime for a modular form $f$ and $v \mid 2$. 
 	The class of $X_v$ in $Br(F_v)$ is given by the symbol
 	$
 	D(-1)^{[F_v:\Q_2]} \cdot (2,D_{K'})_v.
 	$
 \end{theorem} 
 If $k$ is odd, we can predict ramification in terms of nebentypus $\epsilon$. 
 More precisely, we have:  
 \begin{cor} \label{Depsilon}
 	If $p=2$ is a non-dihedral supercuspidal prime for a modular form $f$
 	of odd weight and $v \mid 2$, then we have
 	$
 	[X_v] \sim \epsilon(-1)^{[F_v:\Q_2]} \cdot (2,D_{K'})_v.
 	$
 	 \end{cor}

 \section{Galois representation associated to modular forms and local global compatibility}
\label{galrep}
 For all rational prime $\ell$, we consider a prime $\lambda \mid \ell$ of $E$ 
 and let $E_{\lambda}$ be the completion of $E$ at $\lambda$. 
 For a modular form $f$ as above, Eichler-Shimura-Deligne constructed  a 
 Galois representation  $\rho_{f}=\rho_{f, \lambda}: G_{\Q} \rightarrow \GL_2(E_{\lambda})$. 
 In this paper, we will be using information about  the local Galois representation $\rho_{f, \lambda}|_{G_p}$ with $\ell=p$ called $(p,p)$ Galois representation
 \cite{MR2860430}.

 Let $\A_\Q$ denotes 
 the adeles of $\Q$ and  $\pi$ be the automorphic representation of the adele group 
 $\GL_2(\A_\Q)$ associated to $f$. This has a decomposition as a
 restricted tensor product $\pi=\bigotimes'_p \pi_p$ over all places $p$ (including the infinite primes). 
 Each local component $\pi_p$ is an irreducible admissible representation
 of $\GL_2(\Q_p)$. By the local Langlands correspondence for $n=2$,
 these representations $\pi_p$ are in a bijection with (isomorphism classes of)
 complex $2$-dimensional Frobenius-semisimple Weil-Deligne representations. 
 
 The local global compatibility between these two  Galois representations 
 was proved by Carayol in \cite{MR870690} if $\ell \neq p$. 
 In this paper, we will be using the local global compatibility even for $\ell=p$ 
 proved by Saito \cite{MR1465337} which we describe now. 
 For $p \neq \ell$ (not necessarily $p \nmid N$), the restriction
 $\rho_{f,p}:=\rho_f|_{G_p}$ induces a representation 
 $ {}'\rho_{f,p}: {}'W_p \rightarrow \GL_2(\bar{\Q}_\ell)$
 of the Weil-Deligne group ${}'W_p$ of $\Q_p$.
 Let ${}'\rho_{f,p}^{ss}$ denote its Frobenius semisimplification 
 and let the isomorphism class of Frobenius
 semisimple representation of ${}'W_p$ associated to $\pi_p$ 
 be denoted by ${}'\rho(\pi_p)$.
 The representation ${}'\rho(\pi_p)$ of the Weil-Deligne group of $\Q_p$
 is a pair $(\rho_p(f),N)$ with a representation $\rho_p(f): W_p \to \GL_2(\C)$
 and a nilpotent endomorphism $N$ of $\C^2$ \cite[Section $3$]{MR0347738}.
 In this setting, we have the following diagram: 
 \begin{center}
 	\begin{tikzpicture} [node distance = 2.2cm, auto]
 	\node(A) {$\pi_p$};
 	\node(B) [above of=A] {$\pi_f=\bigotimes'_p \pi_p$};
 	\node(C) [right of=B, node distance = 6cm] 
 	{$\rho(\pi)=\rho_f:G_\Q \to GL_2(\bar{\Q}_\ell)$};
 	\node(E) [below of=C, node distance=1.1cm] 
 	{${}'\rho_{f,p}^{ss}:{}'W_p \to GL_2(\bar{\Q}_\ell)$};
 	\node(D) [below of=E, node distance=1.1cm] 
 	{${}'\rho(\pi_p)=\rho_p:W_p \to GL_2(\C)$};
 	
 	\draw[->] (B) to node {} (A);
 	\draw[<->] (C) to node [swap] {Global Langlands} (B);
 	\draw[->] (C) to node {restriction to $G_p$ induces} (E);
 	\draw[<->] (E) to node {$(*)$} (D);
 	\draw[<->] (A) to node {Local Langlands } (D);
 	\end{tikzpicture}
 \end{center}
 The relation $(*)$ is an isomorphism and it is known for $\ell \neq p$
 by the work of Deligne-Langlands-Carayol. Saito proved the isomorphism $(*)$ even for $\ell =p$.
 For $\ell = p$, the representation $'\rho_{f,p}^{ss}$ is de-Rham and by the 
 landmark paper \cite{MR1779803} 
 can be studied using filtered $(\phi,N)$ modules.

 Recall that for a dihedral supercuspidal prime $p=2$ for $f$,  
 the local Galois representation 
 $
 \rho_f|_{G_2} \sim \Ind_{G_K}^{G_2} \chi,
 $
 with $K|\Q_2$ quadratic.
 In this case, if $N_p=2$, then we show that the extension $K|\Q_2$ is always
 unramified.
 For the character $\chi$ of $G_K$, the usual conductor
 $a(\chi)=\text{min} \{n: \chi(U_K^n)= 1\}$.
 Let $v_2$ be the normalized valuation of $\Q_2^\times$ and 
 $\delta(K|\Q_2),f(K|\Q_2)$ denote the discriminant and the residual degree 
 for $K|\Q_2$ respectively. 
 We now recall the formula \cite[Proposition $4(b)$, p. $158$]{CasselFrohlich} 
 which coincides with the formula for the Artin conductor of a $2$-dimensional
 induced representation of a local Galois group:
 $
 a(\mathrm{Ind}_{G_K}^{G_2} \chi)=v_2(\delta(K|\Q_2)) +f(K|\Q_2)a(\chi).
 $
 This gives
 \begin{eqnarray} 
 N_2 = 
 \begin{cases}
 2 a(\chi), & \quad \text{if} \,\, K|\Q_2 \,\,
 \text{is unramified}, \\
 2+a(\chi), & \quad \text{if} \,\, K|\Q_2 \,\,
 \text{is ramified with discriminant valuation} \,\, 2, \\
 3+a(\chi), & \quad \text{if} \,\, K|\Q_2 \,\,
 \text{is ramified with discriminant valuation} \,\, 3.
 \end{cases}
 \end{eqnarray}
 If $K|\Q_2$ is unramified, $N_2$ always becomes even.
 We see that $N_2=2$ happens in the following cases:
 \begin{enumerate}
 	\item 
 	$K|\Q_2$ is unramified with $a(\chi)=1$ and
 	\item
 	$K|\Q_2$ is ramified with discriminant valuation $2$ and $a(\chi)=0$.
 \end{enumerate}
 The second case cannot occur. Since the algebra $X_f$ is invariant with respect to
 twisting by a Dirichlet character \cite[Proposition $3$]{MR633903},
 without loss of generality one can take $f$ to be minimal 
 in the sense that its level is the smallest
 among all twists $f \otimes \psi$ of $f$ by Dirichlet characters $\psi$.
 Then by \cite[\S $41.4$ Lemma]{MR2234120}, we have $a(\chi) \geq d=2$,
 a contradiction to $a(\chi)=0$ in the second case.

 \begin{lemma} \label{l1}
 	Let $p=2$ be an unramified dihedral supercuspidal prime for $f$ with $N_2=2$.
 	For all $j \in I_W(K)$, we have  $\alpha(j) \in F^\times$.
 \end{lemma}
 
 \begin{proof}
 	Since $N_2=2$, we have $a(\chi)=1$, that is, $\chi|_{U_K^1}=1$. 
 	We know that reciprocity map sends wild inertia group of $K$
 	onto the principal unit group of $K$.
 	Let $\tau=\tau_k \in I_W(K)$ be an element which is mapped to
 	$k \in U_K^1 \subset K^\times$ under the reciprocity map. 
 	Hence, using property $(2)$ of Proposition~\ref{prop1}
 	we obtain $\alpha(\tau) \equiv \chi(k)+\chi^\sigma(k) \equiv 1$
 	mod $F^\times$.
 \end{proof}

 We can realize the nebentypus $\epsilon$ as an idelic character as follows:
 for $x \in \Q_p^\times$, let $[x]$ denote the corresponding element
 $(1,\cdots,x,\cdots,1)$ in $\mathbb{A}_\Q^\times$. The restriction
 of $\epsilon$ to $\Q_p^\times$ is then given by the formula:
 \begin{eqnarray} \label{idelic}
 \epsilon([p^mu])=\epsilon'(p)^m \epsilon_p(u)^{-1},
 \end{eqnarray}
 for $m \in \Z$ and $u \in \Z_p^\times$. From class field theory, we know that
 norm residue map sends $\Q_p^\times \subseteq \mathbb{A}_\Q^\times$ onto a
 dense subset of the decomposition group $G_p$ at $p$. The Galois character
 $\epsilon|_{G_p}$ can also be determined by this fact using the formula
 above.
 \section{Inertial Galois representations} 
 \label{chi}
  In this section, we assume the familiarity of the reader with  \cite{MR2471916}. 
  \subsection{Odd supercuspidal primes}
  For an odd supercuspidal prime, the local Galois representation
  $
  \rho_f|_{G_p} \sim \Ind_{G_K}^{G_p} \chi,
  $
  with $K|\Q_p$ quadratic. Assume $\sigma$ is the generator of $\Gal(K|\Q_p)$.
 Choose a finite extension $L|\Q_p$ with the property that 
  $\rho_f$ is crystalline over $L$ (cf. \cite{MR2471916} for more details) and $L|\Q_p$ is Galois.
  
 Let $\omega_2$ is the fundamental character of level $2$, 
  $\omega$ is the Teichm\"{u}ller character and $\chi_m$ is the character
  having some $p$-power order for $m=1,2$. The inertia type of $\chi$ can be written as follows:
  $
  \chi|_{I_p} = \chi|_{I(L|\Q_p)}=\omega_2^l \cdot \chi_1 \cdot \chi_2$
  (when $K$ is unramified) \cite[Sections~$3.3.2$]{MR2471916}
  and $
  \chi|_{I_K}=\chi|_{I(L|K)}=\omega^l \cdot \chi_1 \cdot \chi_2$
  \cite[Sections~$3.4.2$]{MR2471916}.
  The action of $\sigma$ on these characters are given by the following rule:
  $\omega_2^\sigma=\omega_2^p, \omega^\sigma=\omega,
  \chi_1^\sigma=\chi_1$ and $\chi_2^\sigma=\chi_2^{-1}$.
  
  Since $\chi$ does not extend to $G_p$, we have $\chi \neq \chi^\sigma$
  on $G_K$ which is equivalent to that $\chi \neq \chi^\sigma$ on $I_K$.
  The last condition is equivalent to: 
  either $l \not \equiv 0 \pmod{p+1}$
  or $\chi_2^\sigma \neq \chi_2^{-1}$ (unramified case)
  and $\chi_2^\sigma \neq \chi_2^{-1}$ (ramified case).

 \subsection{Dihedral supercuspidal prime $p=2$}
 In this case, we will see that $\chi|_{I_2}$ can be thought of 
 as a character of an inertia subgroup of a finite Galois extension 
 of $\Q_2$. By a computation similar to  \cite[Sections~$3.3.2$, Sections~$3.4.2$]{MR2471916} for 
 odd primes $p$,  we show that $\chi$ restricted to inertia group 
 can be written as $\chi|_{I_2}=\omega_2^l \cdot \chi_1 \cdot \chi_2$ 
 (in the unramified case) and $\chi|_{I_K}=\omega \cdot \chi_1 \cdot \chi_2$ 
 (in the ramified case). The  results in this section are obtained by generalizing the construction 
 of $\chi$ on the inertia group for $p=2$ following \cite{MR2471916}.
 
	Let $W_2$ (respectively $W_K$) be the Weil group of $\Q_2$ (respectively $K$) and $\rho_2(f)$ be the local representation associated to the local representation $\pi_2$ [cf. Section~\ref{galrep}].
	In this case, the inertia group acts reducibly.
	If it acts irreducibly, then the image of $\rho_2(f)$ becomes an exceptional group. 
	Here, we only concentrate on the dihedral supercuspidal representations.
	To write down the inertia type $\chi|_{I_2}$ or $\chi|_{I_K}$, we recall the structure of the local Galois representation 
	$\rho_2(f)$ following \cite{MR2471916}.
	\subsection{The case $K$ unramified}
	\label{Kunramifiedp2}
	 In this case, $\chi$ is a character
	of $W_K$ which does not extend to $W_2$ and it is finite on $I_2$.  
	Let $\Gal(K|\Q_2)$ is generated by $\sigma$.
	Let $K=\Q_2(\omega)$ be the unique unramified quadratic extension of $\Q_2$
	with $\omega$  a primitive $3$-rd root of unity. We choose a finite 
	extension $L|K$ over which $\rho_2(f)$ becomes crystalline and $L|\Q_2$ 
	is Galois. For an integer $m \geq 1$, let $K^m$ be the unique cyclic 
	unramified extension of $K$ of degree $m$. Consider the polynomial $g(X)=
	\pi X + X^4$, where $\pi$ is a fixed uniformizer of $K$. For a Lubin-Tate 
	module $M$, consider the $\sO_K$-module
	of $\pi^{n+1}$-torsion points
	$$W_g^n:=\text{ker}([\pi^{n+1}]_M)$$
	whose module structure is induced by the formal group attached to $g(X)$.
	Let $K(W_g^n)$ be its field. By local class filed theory, it is a totally 
	ramified abelian extension of $K$ and its
	Galois group 
	\begin{eqnarray} \label{isomorphism1}
	  \Gal(K(W_g^n)|K) \cong U_K/U^{n+1}_K = \F_4^\times 
	  \times \sO_K/ \pi^n,
	\end{eqnarray}
	where $U_K$ and $U_K^{(n+1)}$ denote the units and $(n+1)$-th 
	principal units of $K$ respectively. Furthermore, since $g(X)$ is defined
	over $\Q_2$ (if the uniformizer $\pi$ is chosen from $\Q_2$), the extension
	$K(W_g^n)|\Q_2$ is also Galois.
	
	Consider a finite cyclic extension $F|K$ such that $\chi|_{I_F}$ is trivial. 
	By local class field theory the field $F$ is contained in $K^m K(W_g^n)$, for some $m$
	and $n$ and so $\rho_2(f)$ restricted to its inertia subgroup is trivial. 
	For this reason, we take $L=K^m K(W_g^n)$ over which $\rho_2(f)$ becomes 
	crystalline as $\rho_2(f)$ is trivial on $I_L$ and if the fixed uniformizer
	$\pi$ is chosen to be $2$, the extension $L|\Q_2$ becomes Galois.
	
	\subsubsection{Description of $\Gal(L|\Q_2)$}
	We now describe $\Gal(L|\Q_2)$ in detail. Let $\alpha$ be a root of 
	$g^{(n+1)}(X)$ but not a root of $g^{(n)}(X)$, where $g^{(n)}(X)$ denote the 
	$n$-th iterate of $g(X)$. We have an identification of fields $K(W_g^n)=K(\alpha)$ with  
	$\Q_2(\alpha)|\Q_2$ 
	a totally ramified extension of degree $(2^2-1) \cdot 2^{2n}$ and
	$L|\Q_2(\alpha)$ is an unramified extension of degree $2m$. 
	Let $\sigma$ be a generator of $\Gal(L|\Q_2(\alpha))$ and  its projection to the generator of $\Gal(K|\Q_2)$ is also denoted by $\sigma$ .
	
	Let us write the inertia subgroup of $\Gal(L|\Q_2)$ explicitly; i.e., 
	$\Gal(L|K^m) \simeq \Gal(K(W_g^n)|K)$. Note that $K(W_g^0)= K(\beta)$ with
	$\beta$ a root of $X^3+2=0$. Let $\Delta$ be its Galois group over $K$ 
	which is generated by an element, say $\delta$, of order $3$. It is isomorphic
	to the tame part of the inertia subgroup of $\Gal(L|\Q_2)$. Since the order of 
	$\delta$ and $2$ are relatively prime, $\delta$ can be lifted uniquely to an
	element of order $3$ in $\Gal(K(W_g^n)|K)$, again denoted by $\delta$. 
	The wild part of the inertia subgroup of $\Gal(L|\Q_2)$ is isomorphic to 
	$\Gamma=\sO_K/2^n \cong \Z/2^n \oplus \Z/2^n = <\gamma_1> \oplus
	<\gamma_2>$ with $\gamma_1, \gamma_2$ each have order $2^n$.

	The full inertia subgroup of $\Gal(L|\Q_2)$ is $\Delta \times \Gamma $. This is a
	normal subgroup and it is a direct product of three cyclic groups generated by
	$\delta, \gamma_1$ and $\gamma_2$ respectively. 
	These generators are characterized by the Equ.~(\ref{action1}).
	Since $\sigma^2$ fixes $K(W_g^n)$
	the action of $\sigma$ on $\Gal(L|K^m)$ by conjugation is an involution. 
	Indeed, if $h \in \Gal(L|K^m)$ and $x \in K^m$, then 
	we have $\sigma^2 \cdot h(x)=\sigma^2h\sigma^{-2}(x)=\sigma^2(\sigma^{-2}(x))
	=x=h(x)$ and if $x \in K(W_g^n)$, we have
	$\sigma^2 \cdot h(x)=\sigma^2h\sigma^{-2}(x)=\sigma^2(h(x))=h(x)$.
	This action coincides with the action of $\Gal(K|\Q_2)$ on $\Gal(K(W_g^n)|K))$
	by conjugation. The group $\Gal(K|\Q_2)$ acts on 
	$\sO_K^\times/U_K^{n+1}$ in a natural way.
	Note that
	$
	  \sO_K^\times/U_K^{n+1}
	  =
	  \sO_K^\times/U_K^1 \times U_K^1/U_K^{n+1}
	  \cong 
	  \big( \sO_K/2 \big)^\times 
	  \times  \sO_K/2^n
	  \cong 
	  \F_4^\times \times \sO_K/2^n
	$
	and
	\[
	  \sO_K/2^n = \Z_2[\alpha]/2^n = (\Z_2 \oplus \Z_2 \cdot \alpha)/2^n \cong
	  \Z/2^n \oplus \Z/2^n \cdot \bar{\alpha} 
	  \,\,\, \text{with} \,\,\, \alpha^\sigma= -\alpha.
	\] 
	Let $\rho_K: K^\times \to \Gal(K^{ab}|K)$ be the norm residue map.
	We again denote the restriction $\rho_K|_{\sO_K^\times}$ modulo $U_K^{n+1}$
    by $\rho_K$. This is the isomorphism~(\ref{isomorphism1}).
	We now consider the following commutative diagram \cite[Thoerem~$6.11$]{MR863740}:
	\begin{center}
		$\begin{CD}
		\sO_K^\times/U_K^{n+1}=\F_4^\times \times \sO_K/2^n  @> \rho_K  >> \Gal(F_1|K)  \\
		@VV \sigma V    @VV \sigma^* V  \\
		\sO_K^\times/U_K^{n+1}=\F_4^\times \times \sO_K/2^n  @> \rho_K >>  \Gal(F_1|K),
		\end{CD}$
	\end{center} 
	where $F_1=K(W_g^n)$ and the map $\sigma^*$ is obtained by the conjugated action of 
	$\sigma$ on $\Gal(F_1|K)$. Using the commutativity of the above diagram,
	we have $\rho_K(\sigma(x))=\sigma^{-1}\rho_K(x)\sigma$,
	for all $x \in \sO_K^\times/U_K^{n+1}$. 
	This gives us the following relations:
	\begin{eqnarray}
	\label{delta}
	 \sigma^{-1} \delta \sigma = \delta^2, \quad \sigma^{-1} \gamma_1 \sigma
    	=\gamma_1 \quad \text{and} \quad \sigma^{-1} \gamma_2 \sigma =\gamma_2^{-1}.
   \end{eqnarray}
	
	\subsubsection{Action of $\sigma$} \label{action3}
	By the action $(2.1)$ of \cite{MR2471916}, the character $\chi$ on $I_2$ 
	can be thought of a character $\chi$
	on the inertia subgroup of $\Gal(L|\Q_2)$ which is $\Delta \times \Gamma$
	\cite[Section $3.3.2$]{MR2471916}. 
	Write 
	\begin{eqnarray}\label{character1}
	    \chi|_{I_2} = \chi|_{I(L|\Q_2)}=\omega_2^l \cdot \chi_1 \cdot \chi_2,
	\end{eqnarray}
	where $\omega_2$ is the 
	fundamental character of level $2$ and $\chi_m$ is the character taking
	$\gamma_m$ to a $2^n$-th root of unity $\zeta_m$ for $m=1,2$. Let us
	assume that $\chi_1$ takes $\gamma_1$ to $\zeta_{2^r}$ and $\chi_2$ takes
	$\gamma_2$ to $\zeta_{2^s}$. Here, we denote by $\zeta_{2^r}$ and $\zeta_{2^s}$ 
	a primitive $2^r$-th root of unity and a primitive $2^s$-th root of unity
	respectively, so $r,s \leq n$. Let $\sigma$ be the non-trivial element of 
	the Galois group of $K|\Q_2$ and it acts on the above characters in the 
	following way:
	\begin{eqnarray} \label{action1}
	    \omega_2^\sigma=\omega_2^2, \quad \chi_1^\sigma=\chi_1,
	    \quad \chi_2^\sigma=\chi_2^{-1}.
	\end{eqnarray}
	The condition that $\chi$ does not extend to $W_2$, we have $\chi \neq 
	\chi^\sigma$ on $W_K$ which is further equivalent to that $l \not\equiv 0
	\pmod{3}$ or $\zeta_{2^s} \neq \zeta_{2^s}^{-1}$. 
	Since $\zeta_{2^r}^\sigma=\zeta_{2^r}$ and $\zeta_{2^s}^\sigma=\zeta_{2^s}^{-1}$,
    one can deduce that $r<s$.

\subsection{The case $K$ ramified}
	 Let us now assume that $K|\Q_2$ is a ramified quadratic extension with
	 $\chi$ finite on $I_K$ such that $\chi|_{I_K}$ does not extend to $I_2$. 
	 Let us denote the $\Gal(K|\Q_2)$ by $<\iota>$. Similar to the unramified case, 
	 we find out a Galois 
	 extension $L|\Q_2$ such that $\rho_2(f)|_{I_L}$ is trivial. 
	 
	 \subsubsection{Description of $\Gal(L|\Q_2)$}
	 For an integer
	 $m \geq 1$, let $K^m$ be the unramified extension of $K$ of degree $m$. 
	 For a uniformizer $\pi$ of $K$, let $g(X)=\pi X + X^2$ and as before let
	 $$W_g^n=\{ \alpha \in \mathcal{M}_g \,\, |\,\, \pi^{n+1} \cdot \alpha =0\}. $$
	 Here,  $\mathcal{M}_g$ denote the formal $\sO_K$-module whose underlying set is 
	 the ring of integers of the completion of $\bar{K}$ and its module structure 
	 is induced by the formal group attached to $g$. The field $K_\pi^n=K(W_g^n)$ 
	 is a totally ramified abelian extension of $K$ with Galois group isomorphic to
	 $U_K/U_K^{n+1}=\{1\} \times \sO_K/\pi^n$,
	 where $U_K$ and $U_K^{n+1}$ denote the units and $(n+1)$-th 
	 principal units of $K$ respectively. Note that 
	 $\sO_K/\pi^{2n} \cong \Z/2^n \oplus \Z/2^n$.
	 
	 We now choose a finite cyclic extension $F|K$ such that $\chi|_{I_F}=1$.  
	 As every abelian extension of $K$ is
	 contained in $K^mK_\pi^n$ for some $m,n$ by class field theory, 
	 we have an inclusion of fields
	 $F \subset K^mK_\pi^n$, for some $m,n$. We take a
	 uniformizer $\pi$ of $K$ such that $\pi^\iota = -\pi$
	 (for any lift of $\iota$, again call $\iota$). 
	 The polynomial $g_{-\pi}
	 (X)= -\pi X + X^2$ gives rise to the Lubin-Tate extension $K_{-\pi}^n$ with
	 $(K_\pi^n)^\iota=K_{-\pi}^n$ 
	 which is same as $K_\pi^n$. 
	 Indeed, if $K_\pi^n=K(\alpha)$ then $K_{-\pi}^n=K(-\alpha)$.
	 Thus, the field $K_\pi^n$ is preserved by (a lift of) $\iota \in \Gal(K|\Q_2)$
	 and so $K_\pi^n$ is Galois over $\Q_2$. 
	 
	 For our convenience, set $L=K^{2m}K_\pi^{2n}$. 
	 Then $L|\Q_2$ is a Galois extension containing $F$. In
	 particular $\rho_2(f)|_{I_L}=1$ and $\rho_2(f)$ becomes crystalline over $L$.
	 The description of the Galois group of $L|\Q_2$ is given 
	 using the following exact sequence:
	 \begin{eqnarray*} 
	 	1 \to \Gal(L|K) \to \Gal(L|\Q_2) \to \Gal(K|\Q_2) \to 1 
	 \end{eqnarray*}
	 where
	 $\Gal(L|K)= \Gal(L|K_\pi^{2n}) \times \Gal(L|K^{2m})=<\sigma>
	 \times (\Delta \times \Gamma), \Gal(K|\Q_2)=<\iota>$, with 
	 $\sigma^{2m}=1$ and
	 \begin{center}
	 	$\Delta=\{ 1\}, \,\, \Gamma=\sO_K/\pi^{2n}=
	 	<\gamma_1> \times <\gamma_2>$
	 	with $\gamma_i^{2^n}=1$ for $i=1,2$.
	 \end{center}
	 Here, $\Delta$ and $\Gamma$ are the tame and wild parts of the 
	 inertia group $I(L|K)=\Gal(L|K^{2m})$ respectively.
	 The full inertia subgroup of $\Gal(L|K)$ is a direct product 
	 of two cyclic groups each generated by
	 $\gamma_1$ and $\gamma_2$ respectively.
     These generators are characterized by the Equ.~(\ref{action2})	 
	 \subsubsection{Action of $\iota$} 
	 \label{action4}
	 The inertia $I(L|K)$ is a normal subgroup of $I(L|\Q_2)$ and 
	 the conjugation action of $\iota$ is given by
	 \begin{eqnarray} \label{Action1}
	 	\iota^{-1} \{1\} \iota=\{1\}, \quad 
	 	\iota^{-1}\gamma_1\iota= \gamma_1, \quad 
	 	\iota^{-1} \gamma_2 \iota=\gamma_2^{-1},
	 \end{eqnarray}
	 which can be checked as in the unramified supercuspidal case.
	 As in the previous case we can think of $\chi|_{I_K}$ as a character 
	 of $I(L|K) \cong \Delta \times \Gamma$. 
	 Write 
	 \begin{eqnarray} \label{character2}
	     \chi|_{I_K}=\chi|_{I(L|K)}=\omega \cdot \chi_1 \cdot \chi_2,
	 \end{eqnarray}
	 where $\omega$ is a trivial character, $\chi_m$ is the character taking 
	 $\gamma_m$ to a $2^n$-th root of unity $\zeta_m$ for $m=1,2$. 
	 Let us assume that $\chi_1$ takes $\gamma_1$ to $\zeta_{2^r}$ 
	 and $\chi_2$ takes $\gamma_2$ to $\zeta_{2^s}$. 
	 Here, $\zeta_{2^r}$ and $\zeta_{2^s}$ denote a primitive $2^r$-th 
	 root of unity and a primitive $2^s$-th root of unity
	 respectively and so $r,s \leq n$. The element 
	 $\iota$ acts on the above characters in the following way:
	 \begin{eqnarray} \label{action2}
	     \omega^\iota=\omega, \quad \chi_1^\iota=\chi_1,
	     \quad \chi_2^\iota=\chi_2^{-1}.
	 \end{eqnarray} 
	 The condition $\chi|_{I_K}$ does not extend to $I_2$ 
	 is equivalent to $\zeta_{2^s} \neq \zeta_{2^s}^{-1}$ 
	 and hence $r<s$. Note that there are seven quadratic extensions 
	 $\Q_2(\sqrt{d})$ of $\Q_2$ with $d=-3,-1,3,2,-2,6,-6$. 
	 Among them $\Q_2(\sqrt{-3})$ is unramified 
	 and rest of them are ramified.
 \begin{remark}
 	Note that the above characters $\omega_2^l,\omega, \chi_1$ and $\chi_2$ 
 	are canonically determined by the modular form $f$
 	(more precisely, the actions~\ref{action1} and \ref{action2})
 	as we started with the local representation 
 	canonically attached to $f$
 \end{remark}
 \begin{definition}($\ga_1$-element and $\ga_2$-element) 
 	\label{ga_1ga_2elements}
 	An element of $I_W(K)$ (the wild inertia part of $K$) is called a $\ga_1$-element
 	(resp. $\ga_2$-element) if its projection to $I_W(L|K)$ 
 	is $\ga_1$ (resp. $\ga_2$).
 \end{definition}	 

  \section{Ramifications of endomorphism algebras for odd supercuspidal primes}	
 \label{odd}
 For an odd supercuspidal prime $p$, the local Galois representation is always dihedral 
	and hence induced by a character $\chi$ of an index two subgroup $G_K$ 
	of the local Galois group $G_p$; namely 
	$\rho_f|_{G_p} \sim \text{Ind}_{G_K}^{G_p} \chi$
	with $K$ a quadratic extension of $\Q_p$. 
	The structure of $\chi$ on the inertia group is given
	in the section~\ref{chi}.
 In this section, we give a proof of the results stated in Section~\ref{Statement}
 for odd supercuspidal primes.
 Let  $K$ be an unramified quadratic extension. 
 For $i \in I_K$, let $\bar{i}$ be the projection to $I(L|\Q_p)$ and $\de$ be as above [cf. Section~\ref{chi}].
 We call $\epsilon$ to be  tame at $p$ if the order of $\epsilon_p$
 divides $p-1$.
 \begin{lemma} \label{lem4}
 	If $\epsilon$ is tame at $p$, then $\alpha(j) \in F_v^\times$
 	for all $j \in I_W(K)$.
 \end{lemma}
 
 \begin{proof}
 	Note that $j$ is an element of a pro $p$-group and $p$ is odd.
 	Since $\epsilon$ is tame at $p$, 
 	we must have $\epsilon(j)=1$ and so 
 	$\chi_\gamma^2(j) = \epsilon^{\gamma -1} (j) =1$
 	for all $\gamma \in \Gamma$. By the nature of $j$ and $p$ is odd, 
 	we have $\chi_\gamma (j)=1$ 
 	for all $\gamma \in \Gamma$. This  implies that 
 	$\alpha(j)^{\gamma-1}=1$, for all $\gamma \in \Gamma$ [cf. Equ.~(\ref{1})]. 
 	Hence, we obtain $\alpha(j) \in F^\times$.
 \end{proof}
 Let $s$ be a fixed $(p^2-1)$-th primitive root of unity
 as in \cite{MR3391026} and $K=\Q_p(s)$ 
 is unramified. 
 Recall that $g_s \in \Gal(\bar{\Q}_p|K)$ is an element which is mapped to
 $s \in K^\times$ under the reciprocity map.

 The next lemma shows that
 the hypothesis \textbf{(H)} is same as the condition of \cite[Theorem~$6.1$]{MR3391026}.
 First observe that this condition depends on the choice of $s$.
 By the structure theorem of the local field $K^\times$, we have
	$K^\times = <p> \times <s> \times U_K^{(1)}$ . 
	Let $L$ be as in the beginning of Section~\ref{chi}. 
	By class field theory, the elements of $<s>$ corresponds to the tame part 
	of the inertia group $I(L|\Q_p)$ under the norm residue map. 
  Let $\delta$ be a $p^2-1$-th root of unity as in \cite[Equation 3.3]{MR2471916} (see also Equation~\ref{delta}). Observe that $\delta$ is also a valid choice of $s$. 
\begin{lemma} \label{lem11}
	The assumption $\text{Tr}(\rho_f(g_s)) \neq 0$ in the 
	 \cite[Theorem $6.1$]{MR3391026} is same as \textbf{(H)}.
\end{lemma}

\begin{proof}
	Note that $\text{Tr}(\rho_f(g_s))=\chi(s)+\chi(s)^p$. For the choice of $s=\delta$, 
	this is equivalent to $\chi(\de) + \chi(\de)^p \neq 0$. In other words, 
	$\omega_2^l(\de)+\omega_2^{lp}(\de) \neq 0$ [cf. Section~\ref{chi}].
	Since $\omega_2$ takes value in the $(p^2-1)$-th roots of unity,
	the last condition is same as the condition $l$ is not an odd multiple of $(p+1)/2$.
\end{proof}

\begin{lemma}
 \label{lem12}
Let $p$ be an odd unramified supercuspidal prime
	for $f$ and  satisfying one of the following conditions: 
 	\begin{enumerate}
 		\item 
 		     $p \equiv 1 \pmod{4}$ with $C_p=0$ 
		      		\item
 		     $p \equiv 3 \pmod{4}$ with $C_p = 1$.
 	\end{enumerate}
Then, the condition \textbf{(H)} is satisfied for $p$. 
\end{lemma}

\begin{proof}
	Note that the condition $\text{Tr}(\rho_f(g_s))=0$ is equivalent
	to $\chi(s) + \chi(s)^p=0$, that is,
	\begin{equation}
	\label{eqn1}
	\chi(s)^{p-1}=-1.
	\end{equation}
	
	First consider the case ($1$).
	Write $p=4k+1$, for some $k \in \N$. 
	Since $s^{p+1} \in \Z_p^\times$, using 
	\cite[Equ. $(4)$]{MR3391026} and $C_p=0$, we have
	$\chi(s)^{p+1}=\epsilon_p(s^{p+1})^{-1}=1$. 
	Combining it with (\ref{eqn1}), we get 
	$\chi(s)^2=-1$. Hence, we obtain $\chi(s)=\pm i$. On the other hand, 
	using (\ref{eqn1}) we have that $\chi(s)^{4k}=-1$, a contradiction.

	We now consider the case ($2$). By the same equation of \cite{MR3391026},
	we have $\chi(s)^{p+1}=\epsilon_p(s^{p+1})^{-1}=\eta$,
	where $\eta$ is a $(p-1)$-th root of unity.
	Combining it with (\ref{eqn1}), we get
	$\chi(s)^2=-\eta$.

	First assume that $p=3$. Since $C_3=1$ and $s^4=-1$, we must have $\epsilon_3
	(s^4)=-1$ and so $\eta=-1$. Thus, we deduce $\chi(s)^4=-1$, a contradiction to
	$\chi(s)^2=-1$.

	Now suppose that $p>3$.
	Write $p=4k+3$, for some $k \in \N \setminus \{0\}$.
	Since $\chi(s)^2=-\eta$, we have $\chi(s)= \pm i \cdot \sqrt{\eta}$.
	Again since $p>3$ and $\chi(s)$ is a primitive $2(p-1)$-th 
	root of unity, we must have that $\sqrt{\eta}$ is a primitive 
	$2(p-1)$-th root of unity, say $\zeta_{2(p-1)}$. Thus, we get
	$\chi(s)=\pm i \cdot \zeta_{2(p-1)}$.
	From the equation (\ref{eqn1}), we have
	$(\pm i \cdot \zeta_{8k+4})^{4k+2}=-1$.
	We arrive at a contradiction
	$\zeta_{8k+4}^{4k+2}=1$.
\end{proof} 		  
 Hence, the assumption of \cite[Theorem $6.1$]{MR3391026} is not needed
 for primes stated in the above lemma.

Note that $\omega_2^{p^2-1}(\de)=1$, i.e., $\omega_2^{(p-1)(p+1)/2}(\de)=-1$.
Without \textbf{(H)} we have $\omega_2^{l(p-1)}=-1$ and it is equivalent to
$\omega_2^l(\de)+\omega_2^{lp}=0$. 
Then for $i \in I_T(K)$ with $\bar{i}=\de$,
the last condition is further equivalent to
$\mathrm{trace}(\rho_f(i))=(\chi+\chi^\sigma)(i)=
\omega_2^l(\de)+\omega_2^{lp}(\de)=0$, i.e.,
$\omega_2^l(\de)$ is a primitive $2(p-1)$-th root of unity, say $a$.

\begin{lemma} \label{lem10}
	Let $p$ be an odd unramified supercuspidal prime 
	for $f$ without \textbf{(H)}.
	Suppose that $N_p \geq 3$ and 
	$\epsilon$ is tame at $p$. 
	For all $i \in I_T(K)$,  we have:
	\begin{eqnarray}
	    \alpha(i) \equiv
		\begin{cases}
		    1 \,\, \text{mod} \,\, F_v^\times, & \quad \text{if} 
		    \,\, \bar{i} \,\, \text{is an even power of}\,\, \de, \\
	        a(\zeta_p - \zeta_p^{-1}) \,\, \text{mod} \,\,
	        F_v^\times, & \quad \text{otherwise}. \\
		\end{cases}
	\end{eqnarray}
\end{lemma}

\begin{proof}
	Let $i \in I_T(K)$ be such that $\bar{i}=\de$.
	By above, we deduce that 
	$\mathrm{trace}(\rho_f(i))=
	\omega_2^l(\de)+\omega_2^{lp}(\de)=0$. 
	For even $n$, we have
	$\alpha(i^n) \equiv \omega_2^l(\de^n)+\omega_2^{lp}(\de^n)
	\equiv \text{Tr}_{K|\Q_p}(\de^n)
	\equiv 1$ mod $F_v^\times$.
	
	We now consider odd $n$.
	By \cite[Lemma $4.1$]{MR3391026}, there exists an element 
	$\tau \in I_W(K)$ such that
	$\chi(\tau)=\zeta_p$ and $\chi^\sigma(\tau)=\zeta_p^{-1}$,
	for some primitive $p$-th root of unity $\zeta_p$
	and $\alpha(\tau) \equiv 1$ mod $F^\times$.
	Thus, we deduce that
	$\alpha(i) \equiv \alpha(i\tau) \equiv (\chi+\chi^\sigma)(i \tau)
	\equiv \omega_2^l(\de)(\zeta_p - \zeta_p^{-1}) 
	\,\, \text{mod} \,\, F_v^\times$. 
	Notice that  $\tilde{\alpha}$ is a homomorphism 
	and $a^2(\zeta_p-\zeta_p^{-1})^2 \in F_v^\times$
	by the same lemma. 
	Hence, we obtain $\alpha(i^n) \equiv \alpha(i^{m+1}) \equiv \alpha(i) \equiv
	a(\zeta_p - \zeta_p^{-1})$ mod $F_v^\times$ with $m$ even.
\end{proof}

 Consider the field $F_v^{'}=F_v(b)$ as in Section~\ref{Statement} . We have:
\begin{lemma} \label{lem9}
	Let $p$ be an odd unramified supercuspidal prime 
	for $f$ with $N_p \geq 3$. 
	Assume the hypothesis \textbf{(H)}
	and $\epsilon$ is tame at $p$.
	If $g \in G_K$ and $\alpha(g) \notin (F_v^{'})^\times$,
	then 
	$\alpha(g) \equiv a(\zeta_p-\zeta_p^{-1})$ mod $(F_v^{'})^\times$.
\end{lemma}

\begin{proof}
	For an unformizer $\pi$ of $K$, let $g_{\pi}$ be the image of $\pi$ under Norm residue map.  Note that every element $g \in G_K$ can be written as $g_{\pi}^ni$ for 
	some $n \in \Z$ and $i \in I_K$. 
	We use Lemma \ref{lem10} and the homomorphism $\tilde{\alpha}$ 
	to obtain the result. 
\end{proof}
 Let $p$ be an odd unramified supercuspidal prime with $K \subseteq F_v$.
 Define a function $f$ on $G_v(\subseteq G_K)$ by
 \begin{eqnarray} 
 \label{definef}
 f(g) = 
 \begin{cases}
 1, & \quad \text{if} \,\, \alpha(g) \in (F_v^{'})^\times, \\
 a(\zeta_p-\zeta_p^{-1}), & \quad \text{if} \,\, \alpha(g) \notin (F_v^{'})^\times.
 \end{cases}
 \end{eqnarray}
 We call an element $g$ type $1$ if $\alpha(g) \in (F_v^{'})^\times$,  
 otherwise we call it type $2$.
 If $\epsilon$ is tame at $p$, then we use
 the fact $a^2(\zeta_p-\zeta_p^{-1})^2 \in F_v^\times$
 and Lemma~\ref{lem9}. We see that if $g$ and $h$ both
 are type $1$ elements then $gh$ is also so, but if one of them 
 is of type $1$ and the other one is of type $2$ then their product is an 
 element of type $2$. The product of two type $2$ elements is an element
 of type $1$.
 Thus, we can and do replace the conditions which define the
 function $f$ by a quadratic character $\psi$ in the following 
 way:
 $\psi(g) =1 $, if $\alpha(g) \in (F_v^{'})^\times$ and $\psi(g) =-1 $, 
 otherwise. The function $f$ can be  seen alternatively as:
 \begin{eqnarray} 
 \label{Imp}
 f(g)= 
 \begin{cases}
 1, & \quad \text{if} \,\, \psi(g)=1, \\
 a(\zeta_p-\zeta_p^{-1}), & \quad \text{if} \,\, \psi(g)=-1.
 \end{cases}
 \end{eqnarray}
 The quadratic character $\psi$ on $G_v$ cut out a quadratic 
 extension of $F_v$, namely $F_v(\sqrt{t})$, for some $t \in F_v^\times$. 
 To compute $\mathrm{inv}_v(c_f)$,
 let $\sigma$ be the non-trivial element of $\text{Gal}(F_v(\sqrt{t})|F_v)$. 
 The cocycle table of the $2$-cocycle $c_f$ is given by:
 \begin{center}
 	\begin{tabular}{ |c|c|c|c| } 
 		\hline
 		& $1$ & $\sigma$ \\
 		\hline
 		$1$ & $1$ & $1$ \\
 		\hline
 		$\sigma$ & $1$ & $a^2(\zeta_p-\zeta_p^{-1})^2$ \\ 
 		\hline
 	\end{tabular}
 \end{center}
 which gives the symbol $(t,a^2(\zeta_p-\zeta_p^{-1})^2)_v$.
 Note that the element $t$ has no square root in $F_v^\times$.
    For the next lemma, we assume that  $\sqrt{p^*} \notin F_v^\times$;
     Otherwise we would have the ramified quadratic extension $\Q_p(\sqrt{p^*})
     \subseteq F_v$. As a result, $\alpha(i) \in F_v^\times \,\, \forall \,\,
     i \in I_v$ (cf. Lemma~\ref{lem5}) and so we do not need any auxiliary function
     $f$ in Theorem~\ref{thm5}. 

\begin{lemma}
	\label{cocycleclassKF_v}
	With the above notations, $(t,a^2(\zeta_p-\zeta_p^{-1})^2)_v=1$, i.e.,
    the cocycle class of $c_f$ is trivial.
\end{lemma}
\begin{proof} 
	The element $t$ is unique up to a square in $F_v^\times$ 
	and it is fixed by the kernel of $\psi$.  
	Since $\Frob_v=\Frob_K^{f(F_v|K)}$, the element $g_v=g_\pi^{f(F_v|K)}$ 
	is a fixed Frobenius in $G_v$. 
	Hence, we deduce that  $\alpha(g_v) \in (F_v^{'})^\times$.
	 
	Let $i$ denote the elements of $I_T(F_v)$ such that $\bar{i}=\de$.
	Let $H$ denote the subgroup of $G_v$ generated by the elements of
	$I_W(F_v)$, even power of $i$ and $g_v$.
	We first show that, $H=  \text{ker}(\psi)$. 
	
	Note that
	$
	  \text{ker}(\psi)=\{ g \in G_v \mid \alpha(g) \in (F_v^{'})^\times \}.
	$
	Since $\tilde{\alpha}$ is a homomorphism,
	by Lemmas~\ref{lem4} and~\ref{lem10} we obtain  $H \subseteq \text{ker}(\psi)$.
	Using the homomorphism $\tilde{\alpha}$ again and Lemma~\ref{lem10}, we have 
	$\alpha(i^n) \notin (F_v^{'})^\times$, for all $n \in \Z$ odd 
	and hence it cannot belong to $ \text{ker}(\psi)$.
	Since every element $g \in G_v$ has the form $g=g_v^n i$
	for some $i\in I_v$ and $n \in \Z$, 
	we have $\alpha(g) \equiv \alpha(i)$ mod $(F_v^{'})^\times$.
	Since $I_v$ is a product of its tame part
	and wild part, we have shown  $\text{ker}(\psi) \subseteq H$ and hence $\text{ker}(\psi) =H$.
    
	We now show that $\sqrt{p^*}:=\big( \frac{-1}{p} \big) \cdot p$ 
	is fixed by all the generators of $H$. 
	For all $g \in G_p$, we have $g(\sqrt{p^*})=\sqrt{p^*}$ or $-\sqrt{p^*}$. 
	Let $j\in I_W(F_v)$ be an element of the wild inertia group of $F_v$. 
	Since it is an
	element of a pro-$p$ group and $p$ is odd, we must have 
	$j(\sqrt{p^*})=\sqrt{p^*}$. 
	For all even $n \in \N$, the elements $i^n$ 
	acts on $\sqrt{p^*}$ in a similar way.
	Since $\Frob_K=\Frob_p^2$, the action of $\Frob_K$ 
	and hence the action of $g_v$
	on $\sqrt{p^*}$ is exactly the same as above.
	Hence, we deduce that  $t=p^*$.
	
	We now compute $(p^*,a^2(\zeta_p-\zeta_p^{-1})^2)_v$. 
	First consider $p \equiv 3 \pmod{4}$. Then we have
	$(Nv-1)/2=(p^{f_v}-1)/2 \equiv f_v \pmod{2}$.
	By \cite[Equs. $(15), (16)$ and $(17)$]{MR3391026}, we have
	$\big( \frac{(a^2(\zeta_p-\zeta_p^{-1})^2)'}{v} \big)=
	\big( \frac{(a^2)'}{v} \big) \cdot \big( \frac{(-p)'}{v} \big)$
	and the valuation
	$v(a^2(\zeta_p-\zeta_p^{-1})^2)=2e_v/p-1 \equiv e_v \pmod{2}$.
	Hence, by (\ref{localsymbol}) we get that $(p^*,a^2(\zeta_p-\zeta_p^{-1})^2)_v
	=(-1)^{e_ve_vf_v} \cdot \big( \frac{(-p)'}{v} \big)^{e_v}
	\cdot \big( \frac{(a^2(\zeta_p-\zeta_p^{-1})^2)'}{v} \big)^{e_v}
	=
	(-1)^{e_vf_v} \cdot \big( \frac{\zeta_{p-1}}{v} \big)^{e_v}
	=
	(-1)^{e_vf_v} \cdot \big( \frac{\zeta_{p-1}}{p} \big)^{e_vf_v}
	=
	(-1)^{e_vf_v} \cdot \big( \zeta_{p-1}^{(p-1)/2} \big)
	=
	(-1)^{e_vf_v} \cdot (-1)^{e_vf_v}=1$.
	
	Now assume that $p \equiv 1 \pmod{4}$.
	Using \cite[Lemma~$4.1$]{MR3391026}, we have that $e_v$ is even and
	$\sqrt{p}=\sqrt{p^*} \in \Q_p(\zeta_p+\zeta_p^{-1}) \subseteq F_v$.
	Hence, the symbol $\big( \frac{(p^*)'}{v} \big)=1$.
	As $e_v$ is in the exponent, $(p^*, a^2(\zeta_p+\zeta_p^{-1})^2)_v=1$.
\end{proof}
 
 Suppose that $K \nsubseteq F_v$ (i.e.,
 $G_v \nsubseteq G_K$) with $K|\Q_p$ unramified quadratic.
 For a fixed Frobenius $g_v \in G_v$, 
 the element $\bar{g}_v \in G_v/G_{KF_v}$ is nontrivial. 
 Thus, every element $g\in G_v$ can be written as 
 \begin{eqnarray} \label{Decom}
     g=g_v^nh, \,\, \text{for some} \,\, 
     h\in G_{KF_v} \,\, \text{and} \,\, n\in \{0,1\}.
 \end{eqnarray}
 Note that $n=0$ when $g \in G_{KF_v}(\subseteq G_K)$. 
 Using this decomposition, we extend the function 
 $f$ (\ref{definef}) (defined on $G_{KF_v} \subseteq G_K$) 
 uniquely to $G_v$, call it $F$, as follows:
 $F(g)=f(h)$. 
 The inflation map 
 $
 \Inf :{}_2 \HH^2(G_{KF_v},(\bar{F_v}^\times)^{\Gal(KF_v|F_v)}) \hookrightarrow 
 {}_2 {\HH}^2(G_v,\bar{F_v}^\times)
 $
 sends the cocycle $c_f$ to $c_F$. Since the inflation map is injective and  the class of $c_f$ is trivial, the cocycle class of $c_F$ is trivial.

\subsection{The case $K \subseteq F_v$ or $KF_v|F_v$ unramified quadratic extension}	
\label{case1}
First we determine the value of 
$\alpha$ at the inertia groups.

\begin{lemma} \label{lem5}
	Let $p$ be an odd supercuspidal prime with
    $K \subseteq F_v$. Assume that  $\epsilon$ is tame at $p$. 
	When $p$ is an unramified supercuspidal prime, 
	we also assume \textbf{(H)}.
	For all $\iota \in I_v$, we have $\alpha(\iota) \in F_v^\times$. 
\end{lemma}

\begin{proof}
    In this case, we have  $KF_v=F_v$ and $I_v \subseteq I_K$.
	Every element $\iota \in I_v$ has the form $\iota= i j$ for 
	some element $i$ of the tame part and some element $j$ of the
	wild part of the inertia group $I_v$. 
	
	In the unramified case, we deduce that:
	$\alpha(\iota) = \alpha(ij) \equiv \alpha(i) \equiv \chi(i)+\chi^\sigma(i) 
	\equiv \omega_2^l (i) + \omega_2^{lp} (i) \,\, 
	\text{mod} \,\, F_v^\times$. 
	For ramified supercuspidal primes, we obtain:
	$ \alpha(\iota)=\alpha(ij) \equiv \alpha(i)
	\equiv \omega^l(i) + \omega^l(i)=2 \omega^l(i) 
	\,\, \text{mod} \,\, F_v^\times$.
	The first congruence relation in both cases follows
	from Lemma \ref{lem4} and the definition of the homomorphism 
	$\tilde{\alpha}$,
	and the second one follows from [Proposition~\ref{prop1}, property $(2)$].
	Since $\omega_2^l(i)$ belongs to 
	$K=\Q_{p^2}$, we obtain 
	$\omega_2^l (i) + \omega_2^{lp} (i) = \mathrm{Tr}_{K|\Q_p} 
	\omega_2^l(i) \in \Q_p^\times \subseteq F_v^\times$. Again since 
	$\omega^l$ takes values in the multiplicative group of $(p-1)$-th 
	roots of unity, in both cases, we conclude that 
	$\alpha(\iota) \in F_v^\times$, for all $\iota \in I_v$.
\end{proof}

We now prove Theorem~\ref{nonexceptionalodd} 
when $K \subseteq F_v$ or $KF_v|F_v$ is unramified quadratic.

\begin{theorem} \label{thm5} 
	Let $p$ be an odd supercuspidal prime with
	$K \subseteq F_v$ or $KF_v|F_v$ is unramified quadratic.  
	If $p$ is an unramified supercuspidal prime, 
	we assume \textbf{(H)} unless $N_p\geq 3$.
	Then $X_v \sim m_v$ for $v \mid p$.
\end{theorem}

\begin{proof}
	Since the endomorphism algebra is invariant under twisting 
	\cite[Proposition $3$]{MR633903}, without loss of generality
	one can assume that $\epsilon$ is tame at $p$.
	\begin{enumerate}
	\item 
	Consider $K \subseteq F_v$ with the hypothesis \textbf{(H)}.  
	By the lemma above, 
	$\alpha(\iota) \in F_v^\times$ for all $\iota \in I_v$.
	Using [Proposition \ref{prop1},  part $(1)$] and $\epsilon_p(g) \in \Q_p^\times$
	(as $\epsilon$ is tame at $p$), we obtain
	$\frac{\alpha^2}{\epsilon'}(g) \in F_v^\times$,
	for all $g \in G_v$. By Lemma \ref{lem1} applied to
	$S=\alpha$ and $t=\epsilon'$, we get 
	$
	\mathrm{inv}_v(c_\alpha) = \frac{1}{2} v \Big(\frac{\alpha^2}{\epsilon'}
	(\Frob_v) \Big) \,\, \text{mod} \,\, \Z.
	$
	\item
	Assume $K \subseteq F_v$ with $K$ unramified and $N_p \geq 3$. The previous
	computation works with \textbf{(H)}. 
	Thus, we consider this case
	without the hypothesis \textbf{(H)}.
	
	Note that $G_v \subseteq G_K$ and $I_v \subseteq I_K$. Set $S=\frac{\alpha}{f}$
	on $G_v$ with $f$ as in (\ref{definef}). 
	Since $\alpha(i) \equiv a(\zeta_p - \zeta_p^{-1})$ mod $F_v^\times \,\, \forall \,\,
	i \in I_v$ with $\alpha(i) \notin F_v^\times$ (by Lemmas~\ref{lem4}
	and \ref{lem10}),  we get  $S(i) \in F_v^\times \,\, \forall \,\, i \in I_v$.
	Since $a^2(\zeta_p - \zeta_p^{-1})^2$ and
	$\frac{\alpha^2}{\epsilon'}(g) \in F_v^\times$, we obtain
	$\frac{S^2}{\epsilon'}(g) \in F_v^\times \,\, \forall \,\,
	g \in G_v$. Then by Lemma~\ref{lem1},
	$
	  \mathrm{inv}_v(c_S)=\frac{1}{2} v \Big(\frac{S^2}{\epsilon'}
	  (\Frob_v) \Big) \,\, \text{mod} \,\, \Z
	  =\frac{1}{2} v \Big(\frac{\alpha^2}{\epsilon'}
	  (\Frob_v) \Big) \,\, \text{mod} \,\, \Z.
	$
	The cocycle $c_\alpha$ can be decomposed as $c_sc_f$ with $c_S, c_f$
	are the cocycles corresponding to $S$ and $f$ respectively.
	Note that the cocycle class of $c_f$ is trivial by Lemma~\ref{cocycleclassKF_v} and hence 
	$
	  \mathrm{inv}_v(c_\alpha)=\mathrm{inv}_v(c_S) +\mathrm{inv}_v(c_f)
	  =\mathrm{inv}_v(c_S)
	  =\frac{1}{2} v \Big(\frac{\alpha^2}{\epsilon'}
	  (\Frob_v) \Big) \,\, \text{mod} \,\, \Z.
	$
	\item
	Next assume that $KF_v|F_v$ is unramified quadratic. In this case, we get
	$I_v=I_{KF_v} \subseteq I_K$. The same computation in $(1)$
	works here with \textbf{(H)}. So assume this case without the hypothesis
	\textbf{(H)}.
	
	Define $S=\frac{\alpha}{F}$ on $G_v$ with $F$ as in the previous paragraph
	of Section~\ref{case1}. Since $I_v=I_{KF_v}$, 
	in the decomposition~(\ref{Decom}) for any element of $I_v$,
	we must have $n=0$. 
	By writing the definition of $F$, we deduce  that $\frac{\alpha}{F}
	=\frac{\alpha}{f}$ on $I_v$.
	By the same argument as in $(2)$,
    we see that two conditions of Lemma~\ref{lem1} are satisfied by
    $S$ and $t=\epsilon'$.
	Hence, we obtain
	$
	\mathrm{inv}_v(c_S)=\frac{1}{2} v \Big(\frac{S^2}{\epsilon'}
	(\Frob_v) \Big) \,\, \text{mod} \,\, \Z
	=\frac{1}{2} v \Big(\frac{\alpha^2}{\epsilon'}
	(\Frob_v) \Big) \,\, \text{mod} \,\, \Z.
	$
	Since the cocycle class of $c_F$ is trivial, 
	we deduce that
	$
	\mathrm{inv}_v(c_\alpha)=\mathrm{inv}_v(c_S) +\mathrm{inv}_v(c_F)
	=\mathrm{inv}_v(c_S)
	=\frac{1}{2} v \Big(\frac{\alpha^2}{\epsilon'}
	(\Frob_v) \Big) \,\, \text{mod} \,\, \Z.
	$
    \end{enumerate}	
	For a prime $p'$ introduced before, we have that 
	$\chi_\gamma(\Frob_p)= \chi_\gamma([p])
	\overset{(\ref{idelic})}{=}\chi_\gamma'(p)=
	\chi_\gamma(p')$, where $\chi_\gamma'$ denote the prime-to-$p$ part
	of $\chi_\gamma$.  By a similar computation, we deduce $\epsilon'(\Frob_p)= \epsilon'(p)$.
	Thus, using \eqref{1} we have 
	$
	  \alpha(\Frob_p) \equiv 
	  \alpha(\Frob_{p'}) \equiv a_{p'}\,\, \text{mod} \,\, F^\times
	$,
	where $\Frob_p$ and $\Frob_{p'}$ denote the Frobenii at 
	the primes $p$ and $p'$ respectively.  Hence, we deduce that
	$\alpha(\Frob_v)=\alpha(\Frob_p^{f_v}) 
	\equiv a_{p'}^{f_v} \,\, \text{mod} \,\, F_v^\times$.
	On the other hand, we have $\epsilon'(p)=\epsilon'(p')=\epsilon(p')$. 
	Hence, in all of the above cases we obtain
	
	\begin{equation*} 
	\mathrm{inv}_v(c_\alpha) = \frac{1}{2} v \Big(\frac{\alpha^2}{\epsilon'}
	(\Frob_v) \Big)
	=\frac{1}{2} \cdot f_v \cdot 
	v\Big(\frac{\alpha^2(\Frob_p)}{\epsilon'(p)} \Big)
	=\frac{1}{2} \cdot f_v \cdot
	v(a_{p'}^2 \epsilon(p')^{-1}) \,\, \text{mod} \,\, \Z.
	\end{equation*}
\end{proof}

 \begin{remark}
 	Writing multiplicatively the above formula, we obtain the 
 	same result as in 
 	\cite[Theorems $6.1,6.2$ and $7.1$]{MR3391026}. 
 	When $p \equiv 3 \pmod{4}$  is a ramified supercuspidal prime 
 	with $e_v$ even and $K \nsubseteq F_v$, we have 
 	$
 	[X_v] \sim (-1)^{f_v \cdot v(a_{p'}^2 \epsilon(p')^{-1})}.
 	$
 	Thus, when $f_v$ is even, we deduce $X_v$ is a matrix algebra
 	over $F_v$ which also follows from the formula $(10)$ 
 	of \cite{MR3391026} as $(p^{f_v}-1)/2 \equiv f_v$ mod $2$.
 	We now consider the case where $f_v$ is odd.
 	Since $(p-1)/2$ is odd and it divides $e_v$ \cite[Lemma $4.1$]{MR3391026} 
 	which is even, using \cite[Lemma $7.2$]{MR3391026} we get 
 	\[
 	[X_v] \sim (p, K|\Q_p)^{f_v \cdot v(a_{p''}^2\epsilon(p'')^{-1})}
 	\sim (p, K|\Q_p)^{v(a_{p''}^2\epsilon(p'')^{-1})}.
 	\]
 	Hence, $X_v$ is unramified, when $(p,K|\Q_p)=1$ 
 	which we cannot conclude from the result obtained 
 	in \cite{MR3391026}. When $(p,K|\Q_p)=-1$, 
 	our result matches up with 
 	\cite[Theorem $7.6$]{MR3391026}.
 \end{remark}

 \begin{remark}
 	The quantity that determines 
 	the algebra $X_v$ is independent of the choice of $p'$.
 	For two distinct primes $p'$ and $q'$ satisfying 
 	(\ref{cong1}), one has $\epsilon(p') = \epsilon(q').$
 	Also, $\chi_\gamma(p') = \chi_\gamma(q') \,\, \forall \,\,
 	\gamma \in \Gamma$.
 	Using (\ref{1}) and [Proposition~\ref{prop1}, part ($3$)],
 	we have 
 	$a_{p'}^{\gamma -1} = a_{q'}^{\gamma -1}
 	\,\, \forall \,\, \gamma \in \Gamma.$ Thus, we get
 	$a_{p'} \equiv a_{q'} \,\, \text{mod} \,\, F^\times$ and so
 	$a_{p'}^2\epsilon(p')^{-1} \equiv a_{q'}^2\epsilon(q')^{-1}
 	\,\, \text{mod} \,\, (F^\times)^2$. 
 	Hence, they have the same $v$-adic valuation modulo $2$.
 \end{remark}

\begin{cor} \label{cor1}
	Assume that $K \subseteq F_v$.
	If $p$ is an odd unramified supercuspidal prime (the hypothesis 
	\textbf{(H)} is needed if necessary) or
	$p \equiv 3 \pmod{4}$ is a ramified supercuspidal prime, 
	then $X_v$ is a matrix algebra over $F_v$ .
\end{cor}

\begin{proof}
	For such primes $p$, we have proved that
	$X_v \sim (-1)^{f_v \cdot v(a_{p'}^2\epsilon(p')^{-1})}$, for $v \mid p$.
	When $p$ is an odd unramified supercuspidal prime, the containment
	$K=\Q_{p^2} \subseteq F_v$ implies that $f_v$ is even and the result follows.
	
	When $p \equiv 3 \pmod{4}$ is a ramified supercuspidal prime with $K \subseteq F_v$, 
	we have $e_v$ is even.  On the other hand,  $K \nsubseteq F_v$ if  $e_v$ is odd. 
	If $(p, K|\Q_p)=1$,  we get the result by using \cite[Lemma $7.2$]{MR3391026} and 
	the fact $(p-1)/2$ is odd and it divides $e_v$
	\cite[Lemma $4.1$]{MR3391026}.
	If $(p, K|\Q_p)=-1$, then we have 
	$\sqrt{p} \in K \subseteq F_v$ [cf. Lemma~\ref{imp}].
	Let $g_{\sqrt{p}}$ be an element which is mapped to
	$\sqrt{p} \in K^\times$ and $g_p \in G_p$ be an element
	which is mapped to $p \in \Q_p^\times$ under the reciprocity map.
	We deduce that $g_p=g_{\sqrt{p}}^2$.
	Thus, using [Proposition~\ref{prop1}, property ($1$)] we obtain that 
	$\frac{\alpha^2}{\epsilon'} (g_{\sqrt{p}}) \in F_v^\times$
	and so $\frac{\alpha^2}{\epsilon'} (g_p) \in (F_v^\times)^2$.
	Note that $g_p$ is one of the Frobenius at $p$
	and $g_v=g_p^{f_v}$ is a Frobenius at $v$.
	Thus, the valuation $v \Big(\frac{\alpha^2}{\epsilon'} (g_v) \Big)$ is even. 
	This completes the proof.
\end{proof}

 Let $p$ is an odd unramified supercuspidal prime for $f$ of level $0$ 
 without the hypothesis \textbf{(H)},
 i.e, a bad prime.  For $i \in I_T(K)$ with $\bar{i}=\de$, we have
 $\omega_2^{l(p-1)}(\de)=-1$ and so $\mathrm{trace}(\rho_f(i))=0$ as before.
 By Lemma~\ref{lem10}, we have $\alpha(i^2) \in F_v^\times$ and  we write
 $\alpha(i)=\sqrt{t(i)} c(i)$ for some $t(i), c(i) \in F_v^\times$.
 Consider two elements $i, j \in I_T(K)$ with $\bar{i}=\bar{j}=\de$.
 Since $c_\alpha(i,j) \in F_v^\times$ and $\alpha(ij) \in F_v^\times$
 (cf. Lemma~\ref{lem10}), we must have $\sqrt{t(i)} \equiv \sqrt{t(j)}$ mod $F_v^\times$.
 For some fixed $c \in F_v^\times$, we have $\sqrt{t(i)} \equiv \sqrt{c}$ mod $F_v^\times$.
 By Lemmas~\ref{lem4} and \ref{lem10}, we have $\alpha(i) \in F_v^\times$ for all $i \in I_K$  except those for which  $\bar{i}$ is an odd power of $\de$.
 For $i \in I_T(K)$ with $\bar{i}=\de$, let us assume:
 \begin{eqnarray} \label{definitionofc}
     \alpha(i) \equiv \sqrt{c} \,\, \text{mod} \,\, F_v^\times,\quad c \in F_v^\times.
 \end{eqnarray}
 Consider an integer $n_v$ modulo $2$ as in (\ref{c}) when $p$ is a bad prime.

\begin{theorem}  
	Let $v \mid p$ be a place of $F$ with 
	$p$  a ``bad" level zero unramified supercuspidal prime.
	The ramification of $X_v$  is determined by the parity of $m_v+n_v$. 
\end{theorem}
\begin{proof}
	We will proceed the same way as before. When $K \subseteq F_v$,
	consider the function on $G_v$ defined by:
	$f(g) = 1$, if $\alpha(g) \in (F_v^{'})^\times$ and
	$f(g)=\sqrt{c}$, if $\alpha(g) \notin (F_v^{'})^\times$.
	Consider the extension
	$F_v(\sqrt{t})|F_v$ cut out by the quadratic character obtained
	from the conditions that define $f$. 
	By a computation as in the previous cases, the cocycle class of $c_f$ is determined by the symbol $(t,c)_v$,

	If $K \nsubseteq F_v$, we extend
	this function $f$ uniquely to $G_v$, call it $F$.  As above, observe that both
	$c_f$ and $c_F$ have the same Brauer class. Define a function $\alpha'$ on $G_v$ as follows:
	\begin{eqnarray} 
	\label{Hprimes}
	\alpha'= 
	\begin{cases}
	\frac{\alpha}{f}, & \quad \text{if} \,\,K \subseteq F_v, \\
	\frac{\alpha}{F}, & \quad \text{if} \,\, K \nsubseteq F_v.
	\end{cases}
	\end{eqnarray}
	Then the assumptions of Lemma~\ref{lem1} will  be satisfied 
	by $S=\alpha'$ and $t=\epsilon'$ and we get the result.
\end{proof}
\subsection{The case $KF_v|F_v$ is ramified}
 This case will happen only if $p$ is an odd ramified supercuspidal 
 prime with $p \equiv 3 \pmod{4}$ and $e_v$ odd. 
 For any quadratic extension $L_1|L_2$ and $x \in L_2^\times$, 
 the symbol $(x, L_1|L_2)= 1$ or $-1$ according as $x$ is a norm  
 of an element of $L_1$ or not.

 In the ramified case, the possibilities for $K$ are $\Q_p(\sqrt{-p})$
 and $\Q_p(\sqrt{-p\zeta_{p-1}})$ depending on $(p, K|\Q_p)= 1$ or $-1$
 respectively. We can choose $\pi=\sqrt{-p}$ or 
 $\sqrt{-p \zeta_{p-1}}$ as a uniformizer of $K$ and write
 $K=\Q_p(\pi)$. For any lift $\sigma$ of the generator
 of Gal$(K|\Q_p)$ to $G_p$, we have $\pi^\sigma=-\pi$
 and $N_{K|\Q_p}(\pi)=-\pi^2$.

For a field $L$, let $\mathcal{O}_L^\times$ be the ring of units inside $\mathcal{O}$. 
 Since $KF_v|F_v$ is a ramified quadratic extension,
 $N_{KF_v|F_v}(\mathcal{O}_{KF_v}^\times)=\mathcal{O}_{F_v}^{\times 2}$. 
 Let $\pi_v$ be a fixed uniformizer in 
 $N_{KF_v|F_v}((KF_v)^\times) \subseteq F_v^\times$.
 Writing $a= \pi_v^{v(a)} \cdot a' \in F_v^\times$, we have
 $\big(\frac{a'}{v}\big)=(a, KF_v|F_v)$. 
 
Note that $f_v$ is odd in our case. For these primes, we have 
 \begin{eqnarray} \label{norm}
   (-1,KF_v|F_v)=
   \Big(\frac{-1}{v} \Big)=\Big(\frac{-1}{p} \Big)^{f_v}=(-1)^{f_v}=-1.
 \end{eqnarray}
   Otherwise, $X_v$ is a matrix algebra
 over $F_v$ [cf. Remark~\ref{extrasymboltrivial}]. 
 Since $N_{KF_v|F_v}(\sqrt{-p\zeta_{p-1}})=p\zeta_{p-1}$ and 
 $N_{KF_v|F_v}(\sqrt{-p})=p$, we deduce that
 \begin{equation} \label{residuesymbol}
 \Big(\frac{(\pi^2)'}{v}\Big) =
     \begin{cases}
        \Big(\frac{(-p)'}{v}\Big)= \Big(\frac{-1}{v}\Big)
        (p, KF_v|F_v)= \Big(\frac{-1}{v}\Big), 
        \quad \text{if} \,\, (p,K|\Q_p)=1,  \\
        \Big(\frac{(-p\zeta_{p-1})'}{v}\Big)=\Big(\frac{-1}{v}\Big)
        (p\zeta_{p-1}, KF_v|F_v)=\Big(\frac{-1}{v}\Big),
        \quad \text{if} \,\, (p,K|\Q_p)=-1.
     \end{cases}
 \end{equation}  
	
\begin{lemma} \label{d}
    For all $\iota \in I_v \setminus I_{KF_v}$ and 
    $\alpha(\iota) \notin F_v^{\times}$, we have 
    $\alpha^2(\iota) \equiv d$ mod $F_v^{\times^2}$,
    for some fixed  $d \in F_v^\times$. 
\end{lemma}

\begin{proof}
	Let us consider an element $i\in I_v \setminus I_{KF_v}$ which we fix now.
	Since $i^2 \in I_{KF_v} \subseteq I_K$, we have
	$\alpha^2(i) \in F_v^\times$ [cf. Lemma~\ref{lem5}].
	Hence, $\alpha(i) \equiv \sqrt{d}$ mod $F_v^\times$
	for some $d \in F_v^\times$.
	Any element $\iota \in I_v \setminus I_{KF_v}$ can be written as
	$\iota=ij$ for some $j \in I_{KF_v}$. Using Lemma~\ref{lem5} and
	the homomorphism $\tilde{\alpha}$, we have
	$\alpha(\iota) \equiv \alpha(i) \alpha(j) \equiv \alpha(i)
	\equiv \sqrt{d}$ mod $F_v^\times$.
\end{proof}

 We show that the value of the constant $d$ is $a_{p^{''}}^2$.
 Let $[\,\,\,]_v:F_v^\times \rightarrow G_v^{\mathrm{ab}}$ be the
 usual norm residue map.

\begin{lemma}
    As an element of the Galois group, we have  $i=[-1]_v \in G_v \setminus G_{KF_v}$.
    Moreover, the value of the map $\alpha$  at $i$ is given by: 
    $
    \alpha(i) \equiv a_{p''} \pmod {F_v^\times}.
    $
\end{lemma}
	
\begin{proof}
   As the norm residue map is surjective, we need to show that
	$[-1]_v \neq [x]_{KF_v}$  for any $x \in (KF_v)^\times$. 
	Suppose towards a contradiction that
	$[-1]_v = [x]_{KF_v}$, for some $x \in (KF_v)^\times$. 
	Let $\rho_{KF_v}=[\,\,\,]_{KF_v}$ and $\rho_v=[\,\,\,]_v$ be the norm reciprocity maps. Recall, the following commutative diagram from the class field theory: 

\begin{equation}
\label{diagram}
\begin{tikzcd}
(KF_v)^\times \arrow[r,"\rho_{KF_v}"] \arrow[d,"N_{KF_v|F_v}"] & G_{KF_v}^{\mathrm{ab}}  \arrow[d,"\text{incl}"] \\
 F_v^\times \arrow[r,"\rho_v"] & G_v^{\mathrm{ab}}.
\end{tikzcd}
\end{equation}
From the above diagram,
	we have $[x]_{KF_v}=[N(x)]_v$ and so $[-1]_v=[N(x)]_v$. 
	We  write $-1=N(x) y$ for some
	$y \in N_{\bar{F}_v|F_v}(\bar{F}_v^\times) =  
	\bigcap_{F_v \subset L \,\, \text{finite}}
	N_{L|F_v}(L^\times).$
	As $KF_v$ is a finite extension of $F_v$,  we deduce  $-1$ is a norm of some element of $(KF_v)^\times$, a contradiction to (\ref{norm}). 
	Since the norm residue map sends the unit group 
	of $F_v$ onto the inertia subgroup of $G_v$, the element $i=[-1]_v \in {I_v} \backslash {I_{KF_v}}$. 
	
	Note that $i=[-1]_v \in G_v(\subseteq G_p)$ is one of the several elements 
	that maps to $-1 \in \Q_p^\times$ under the reciprocity map.
	For all $\ga \in \Ga$, we obtain
	\[
	  \alpha(i)^{\gamma-1} = \chi_\gamma(i)=\chi_\gamma([-1])
   	  \overset{(\ref{idelic})}{=} \chi_{\gamma,p}(-1)^{-1}   
	  = \chi_{\gamma,p}(-1) 
	  = \chi_{\gamma,p}(p^{''})^{\frac{p-1}{2}}
	  = \big( \alpha(\Frob_{p^{''}})^{\gamma-1} \big)^{\frac{p-1}{2}}. 
	\]
	  We deduce that  $\alpha(i) \equiv a_{p^{''}}^{\frac{p-1}{2}}  \,\, \text{mod} 
	\,\, F_v^\times.$ 
	Using the property $(1)$ of Proposition \ref{prop1} and  $\alpha(\Frob_{p^{''}}) \equiv a_{p^{''}} \pmod {F_v^\times}$, we have $\epsilon(p^{''})=\epsilon_p(p^{''})
	\equiv a_{p^{''}}^2 \pmod {F_v^\times}$. 
	Since $p^{''}$ has order $(p-1)$
	in $(\Z/p^{N_p}\Z)^\times$, we have $a_{p^{''}}^2 \in F_v^\times$. 
	As $p \equiv 3 \pmod 4$ and $(p-1)/2$ is odd, we deduce that	$\alpha(i) \equiv a_{p^{''}}$ mod $F_v^\times$.
\end{proof}

\begin{lemma} \label{imp}
	If $p \equiv 3 \pmod{4}$ and $(p,K|\Q_p)=-1$, then we have 
	$a_{p^{''}}^2u \in F_v^{\times 2}$,
	for some unit $u \in \mathcal{O}_v^\times$.
\end{lemma}

\begin{proof}
	For an odd prime $p$, the two ramified
	quadratic extensions of $\Q_p$ are $\Q_p(\sqrt{-p})$ and
	$\Q_p(\sqrt{-p\zeta_{p-1}})$ up to an isomorphism.
	Note that $\Q_p({\sqrt{p}})$ is always a ramified quadratic extension of $\Q_p$
	and $-1$ has no square root modulo $p$ for primes $p \equiv 3 \pmod{4}$. 
	Since when $(p,K|\Q_p)=-1$, the only possibility for $K$ is $\Q_p(\sqrt{p})=\Q_p(\sqrt{-p\zeta_{p-1}})$.

	We obtain $s'=\sqrt{-\zeta_{p-1}} \in K$ and
	let $g_{s'} \in G_K$ be an element which is mapped to
	$s' \in K^\times$ under the reciprocity map.
	We have a following equality:
	\[
	\chi_\gamma(g_{s'})=\chi_\gamma([N_{K|\Q_p}(s')])
	=\chi_\gamma([\zeta_{p-1}])
	\overset{(\ref{idelic})}{=}\chi_{\gamma,p}(p^{''})^{-1}
	=\chi_\gamma(p^{''})^{-1}.
	\] 
	Using (\ref{1}), we deduce  that $\alpha(g_{s'}) \alpha(\Frob_{p^{''}})
	\in F_v^\times$ and hence
	$\alpha(g_{s'}) \cdot a_{p^{''}} \in F_v^\times$.
	We now claim that $\alpha(g_{s'}) \equiv u$ mod $F_v^\times$,
	for some unit $u \in \mathcal{O}_v^\times$. 
	Since $s'$ is a root of unity, the element $g_{s'} \in I_T(K)$ by class field theory.
	For an element $i \in I_T(K)$, we know that
	$\alpha(i) \in F_v^\times$ and
	$\alpha(i) \equiv \omega(i)$ mod $F_v^\times$
	[cf. Lemma~\ref{lem5}].
	Since $\omega$ takes values in the $(p-1)$-th
	roots of unity, we get the result.
\end{proof}

 We now prove Theorem~\ref{extrathm} when $KF_v|F_v$
 is a ramified quadratic extension. 
\begin{proof}
	Define a function $f$ on $G_v$ by
	\begin{eqnarray} \label{ramf}
	    f(g) = 1, \,\, \text{if} \,\, g \in G_{KF_v} \,\,
	    \text{and} \,\, f(g)=a_{p^{''}}, \,\, \text{if} \,\, g \in G_v \backslash G_{KF_v}.
	\end{eqnarray}
   Note that $KF_v=F_v(\pi)$. 
   Denote the image of $g \in G_v$ 
    under the projection in $G_v/G_{KF_v}=\Gal(F_v(\pi)|F_v)$
   by $\bar{g}$. 
   We now consider the  function $F$ on $\Gal(KF_v|F_v)$:
   \begin{eqnarray}\label{ramF}
       F(g) =1, \,\, \text{if} \,\,\overline{g}=1 \,\,
       \text{and} \,\,  F(g)=a_{p^{''}},\text{if} \,\, \overline{g} \neq 1.
   \end{eqnarray}
   
   Using equations~(\ref{ramf}) and (\ref{ramF}), one can check that
   $
   c_F(\bar{g}, \bar{h})=c_f(g,h).
   $
   In other words, we deduce that the inverse of the inflation map
   $\Inf:\HH^2(F_v(\pi)|F_v) \hookrightarrow \HH^2(\bar{F_v}|F_v)$ 
   sends $c_f$ to $c_F$. Let $\sigma$ be the non-trivial element of  
   $\Gal(F_v(\pi)|F_v)$. The  cocycle table of $C_F$  is given by:
	\begin{center}
		\begin{tabular}{ |c|c|c|c| } 
			\hline
			& $1$ & $\sigma$ \\
			\hline
			$1$ & $1$ & $1$ \\
			\hline
			$\sigma$ & $1$ & $a_{p^{''}}^2$ \\ 
			\hline
		\end{tabular}
	\end{center}
	
	\noindent
	which gives the symbol $(\pi^2, a_{p^{''}}^2)_v$.
	Using the above inflation map we get both $c_f$ and $c_F$
	have same class in their respective Brauer groups.
	Define an integer $n_v$ mod $2$ by
	$(-1)^{n_v}=(\pi^2, a_{p^{''}}^2)_v$. 
	 
	Let $\alpha'(g)=\frac{\alpha(g)}{f(g)}$ on $G_v$. Then the cocycle
    $c_\alpha$ can be decomposed as $=c_{\alpha'} c_f$.
	The two conditions of Lemma \ref{lem1} are satisfied by
	$S=\alpha'$ and $t=\epsilon'$. Thus, we obtain:
	$
		\mathrm{inv}_v(c_{\alpha'})= \frac{1}{2} \cdot 
		v\Big(\frac{\alpha'^2}{\epsilon'}(\Frob_v)
		\Big) 
		= 
		\frac{1}{2} \cdot 
		v\Big(\frac{\alpha^2}{\epsilon'}(\Frob_v)
		\Big)
		=
		\frac{1}{2} \cdot f_v \cdot
		v(a_{p'}^2 \epsilon(p')^{-1}) \,\, 
		\text{mod} \,\, \Z,
	$
	as before.
	Hence,  we deduce that:
	$
	    \mathrm{inv}_v(c_\alpha) = \mathrm{inv}_v(c_{\alpha'})+\mathrm{inv}_v(c_f)  
	    = \frac{1}{2} \cdot \big(f_v \cdot
	    v(a_{p'}^2 \epsilon(p')^{-1}) + n_v \big) 
	    \text{mod} \,\, \Z$.
\end{proof}
	Multiplicatively, we can write the above as
    $[X_v] \sim (-1)^{f_v \cdot v(a_{p'}^2 \epsilon(p')^{-1})} 
	\cdot (\pi^2,a_{p''}^2)_v$.
	The following lemma will complete the proof which is a
	simplification of this product depending upon the value of $(p, K|\Q_p)$.

\begin{lemma}
	\label{equiv}
	If $KF_v |F_v$ is a ramified quadratic extension, then the ramification formula is
	given by:
	\[
	[X_v] \sim  \big((-1)^ka_{p^{''}}^2 \epsilon(p^{''})^{-1}, KF_v|F_v \big).
	\]
\end{lemma}

\begin{proof}
   Since $p \equiv 3 \pmod{4}$, we get 
   $(Nv-1)/2=(p^{f_v}-1)/2 \equiv f_v \pmod{2}$.
   Recall that  both $v(\pi^2)=e_v$
   and $(p-1)/2$ is odd that divides $e_v$ \cite[Lemma $4.1$]{MR3391026}.
   We have an equality of symbols:
   $
   	(\pi^2,a_{p^{''}}^2)_v
   	\overset{(\ref{localsymbol})}{=}
   	(-1)^{v(\pi^2)v(a_{p^{''}}^2)(p^{f_v}-1)/2} 
   	\Big(\frac{(a_{p^{''}}^2)'}{v} \Big)^{v(\pi^2)}
   	\Big(\frac{(\pi^2)'}{v} \Big)^{v(a_{p''}^2)} 
   	\overset{(\ref{residuesymbol})}{=}
   	(-1)^{f_v \cdot v(a_{p^{''}}^2)} 
   	\Big(\frac{(a_{p^{''}}^2)'}{v} \Big) 
   	\Big(\frac{-1}{v} \Big)^{v(a_{p^{''}}^2)} 
   	= 
   	\Big(\frac{(a_{p^{''}}^2)'}{v} \Big).
   $
   
   When $(p,K|\Q_p)=1$, using \cite[Lemma $7.2$]{MR3391026} we deduce that:
   \begin{eqnarray*}
   	  [X_v] &\sim& (-1)^{f_v \cdot v(a_{p'}^2 \epsilon(p')^{-1})} 
   	  \cdot (\pi^2,a_{p^{''}}^2)_v  \\
   	  &=& (-1)^{e_v f_v(k-1)}
   	  (-\epsilon_p(-1))^{2e_vf_v/(p-1)}
   	  \cdot \Big(\frac{(a_{p^{''}}^2)'}{v} \Big) \\
   	  &=& (-1)^{(k-1)f_v}(-\epsilon_p(-1))^{f_v} 
   	  \cdot \Big(\frac{(a_{p''}^2)'}{v} \Big) \\
   	  &=& (-1)^{kf_v}  (\epsilon_p(p'')^{(p-1)/2})^{f_v}  
   	  \cdot \Big(\frac{(a_{p^{''}}^2)'}{v} \Big) \\
   	  &=& (-1)^{kf_v} \Big( \frac{\epsilon(p'')^{(p-1)/2}}{p} 
   	  \Big)^{f_v} \cdot \Big(\frac{(a_{p^{''}}^2)'}{v} \Big) \\
   	  &=& (-1)^{kf_v} \Big(\frac{\epsilon(p^{''})}{p} \Big)^{f_v} 
   	  \cdot \Big( \frac{(a_{p^{''}}^2)'}{v} \Big) \\
   	  &=& (-1)^{kf_v} \Big(\frac{(\epsilon(p^{''}))'}{v} \Big) 
   	  \cdot \Big( \frac{(a_{p^{''}}^2)'}{v} \Big) 
   	  \sim \big((-1)^k \epsilon(p^{''})^{-1}, KF_v|F_v \big) 
   	  \cdot \big(a_{p^{''}}^2, KF_v|F_v \big)  \\
   	  &=&  \big((-1)^ka_{p^{''}}^2 \epsilon(p^{''})^{-1}, KF_v|F_v \big). 
   \end{eqnarray*}
   On the other hand for $(p,K|\Q_p)=-1$, using \cite[Lemma $7.2$]{MR3391026} we have:
   
   \begin{eqnarray*}
   	  [X_v] &\sim& (-1)^{f_v \cdot v(a_{p'}^2 \epsilon(p')^{-1})} 
   	  \cdot (\pi^2,a_{p^{''}}^2)_v  \\
   	  &=& (-1)^{e_v f_v(k-1)}
   	  (-\epsilon_p(-1))^{2e_vf_v/(p-1)}
   	  \cdot (-1)^{f_v \cdot v(a_{p^{''}}^2\epsilon(p^{''})^{-1})}
   	  \cdot \Big(\frac{(a_{p^{''}}^2)'}{v} \Big) \\
   	  &=& \big((-1)^ka_{p^{''}}^2 \epsilon(p^{''})^{-1}, KF_v|F_v \big)
   	  \cdot (-1)^{f_v \cdot v(a_{p^{''}}^2\epsilon(p^{''})^{-1})} \\
   	  &\overset{(\ref{imp})}{=}& 
   	  \big((-1)^ka_{p^{''}}^2 \epsilon(p^{''})^{-1}, KF_v|F_v \big).
   \end{eqnarray*}
\end{proof}

 \begin{remark} \label{extrasymboltrivial}
 	In general, when $KF_v|F_v$ is ramified quadratic, we have
 	$[X_v] \sim (-1)^{f_v \cdot v(a_{p'}^2\epsilon(p')^{-1})} 
 	\cdot (\pi^2,d)_v$ with $d$ as in Lemma~\ref{d}.
 	If $f_v$ is even, then by (\ref{localsymbol}), the symbol $(\pi^2,d)_v=1$
 	as $(p^{f_v}-1)/2 \equiv f_v \pmod{2}$.
 	Hence, $X_v$ is unramified.
 \end{remark}
\section{Ramifications for primes lying above dihedral supercuspidal prime  $p=2$}
\label{p=2}
 For $i \in I_K$, let $\bar{i}$ denote the projection to the inertia
 subgroup $I(L|K)$[cf. Section~\ref{chi}]. The following lemma will give information 
 about $\alpha$ on the inertia group $I_K$.
 
 \begin{lemma} \label{lem8}
 	Let $p=2$ be a dihedral supercuspidal prime for $f$. 
 	For all $i \in I_K$, we have 
 	\begin{eqnarray*}
 		\alpha(i) \equiv
 		\begin{cases}
 			1 \,\, \text{mod} \,\, F_v^\times, & \quad \text{if} \,\, 
 			i \in I_T(K), \\
 			\zeta_{2^r} \,\, \text{mod} \,\, F_v^\times, & \quad \text{if} \,\,
 			\bar{i} \,\, \text{is an odd power of} \,\, \ga_1,\\
 			\zeta_{2^s}+\zeta_{2^s}^{-1} \,\, \text{mod} 
 			\,\, F_v^\times, & \quad \text{if}
 			\,\, \bar{i} \,\, \text{is an odd power of} \,\, \ga_2 \text{ and }\zeta_{2^s}+ \zeta_{2^s}^{-1} \neq 0.\\
 		\end{cases}
 	\end{eqnarray*}
 	 Furthermore, $\alpha(i^{2k}) \equiv 1 \,\, \text{mod} \,\, F_v^\times$ for all  $k \in \Z$.
 \end{lemma}
 
 \begin{proof}
    Consider the case  $K$ is unramified.
	By the part $(2)$ of Proposition~\ref{prop1} and 
	using (\ref{character1}), (\ref{action1}) we have
    $
	  \alpha(i) \equiv \omega_2^l(\bar{i})+\omega_2^{2l}(\bar{i}) \,\,
      \text{mod} \,\, F_v^\times $
    for all $i \in I_T(K)$. Since $\omega_2$ takes values in the third roots of 
	unity, we have $\alpha(i) 
	\equiv 1 \,\, \text{mod} \,\, F_v^\times$.

	In the ramified case, we obtain
	$\alpha(i) \equiv \omega(\bar{i})+\omega(\bar{i}) 
	\equiv 1 \,\, \text{mod} \,\, F_v^\times $.
	Let $j_1$ and $j_2$ be a $\ga_1$-element and a $\ga_2$-element respectively
	[cf. Definition~\ref{ga_1ga_2elements}].
	Then $\alpha(j_1) \equiv \chi_1(\gamma_1)+\chi_1^\sigma
	(\gamma_1) \equiv \zeta_{2^r} \,\, \text{mod} \,\, F_v^\times$ and 
	$\alpha(j_2) \equiv \chi_2(\gamma_2) + \chi_2^\sigma(\gamma_2)
	\equiv \zeta_{2^s}+\zeta_{2^s}^{-1} \,\, \text{mod} \,\, F_v^\times.$

	Since $\alpha(j_2)  \equiv \zeta_{2^s}+\zeta_{2^s}^{-1} \,\, \text{mod}
	\,\, F^\times$, using [Proposition~\ref{prop1}, property $(1)$] we have that
	 $\epsilon(j_2)= a(\zeta_{2^s}+\zeta_{2^s}^{-1})^2$ for some $a \in 
	F^\times$. Recall that $\epsilon(j_2)$ is a root of unity. 
	Thus, we obtain $\epsilon(j_2)\in \{\pm 1\}$ and so $(\zeta_{2^s}+\zeta_{2^s}^{-1})^2 \in
	F^\times$.  In turn, this  implies $\zeta_{2^{s-1}}+\zeta_{2^{s-1}}^{-1} \in 
	F^\times$. As $r<s$, we must have $\zeta_{2^r}+\zeta_{2^r}^{-1} \in F^\times$.
	Note that the field $F(\zeta_{2^r})$ inside $E$ has degree $2$ over both the fields
	$F=F(\zeta_{2^r}+\zeta_{2^r}^{-1})$ and $F(\zeta_{2^{r-1}})$. Since
	$F \subseteq F(\zeta_{2^{r-1}})$, we conclude that $\zeta_{2^{r-1}} \in F_v^\times$. 
	We get the desired result using the homomorphism $\tilde{\alpha}$.
	The last statement follows from the observation $\zeta_{2^{r-1}}, (\zeta_{2^s}+\zeta_{2^s}^{-1})^2
	\in F_v^\times$. 
\end{proof}
 
 \begin{lemma} \label{lem7}
 	If $p=2$ is a dihedral supercuspidal prime for $f$,
 	then $\epsilon$ is $F^\times$-valued 
 	on $I_W(K)$.
 \end{lemma}
 
 \begin{proof}
 	For an extra twist $(\gamma, \chi_\gamma)$,
 	we have $\rho_{f^\gamma} \sim \rho_f \otimes \chi_\gamma $. 
 	By restricting the representation to $G_2$, we obtain:
 	$$\begin{pmatrix}
 	\chi^\gamma & 0  \\
 	0 & (\chi^\sigma)^\gamma
 	\end{pmatrix}
 	\sim
 	\begin{pmatrix}
 	\chi \chi_\gamma & 0  \\
 	0 & \chi^\sigma \chi_\gamma
 	\end{pmatrix}.
 	$$
 	Equating deteminants on both sides, we get  $(\chi \chi^\sigma)^\gamma = 
 	\chi \chi^\sigma \chi_\gamma^2$ for all $\gamma \in \Gamma$.
 	Since $\chi_\gamma^2=\epsilon^{\gamma-1}$, 
 	the quantity $\frac{\chi \chi^\sigma}{\epsilon}
 	\in F^\times$ and so $\frac{\chi_1^2}{\epsilon} \in F^\times$
 	on the wild inertia group $I_W(K)$. We 
 	get the lemma as $\chi_1^2 \in F^\times$.
 \end{proof}
 
 \begin{lemma}
 	Assume that $K \subseteq F_v$ with $\zeta_{2^s}+\zeta_{2^s}^{-1}=0$. Let $j_2$ be a $\ga_2$-element. Then $\chi_\ga$ becomes unramified (or ramified) depending
 	on $\chi_\ga(j_2)=1$ (or $-1$), for all $\ga \in \Ga$.
 \end{lemma}		
 
 \begin{proof}
 	By assumption, we have $s=2$.
 	Since $\epsilon$ is $F^\times$-valued on $I_W(K)$ by Lemma~\ref{lem7},
 	we have $\chi_\ga^2(j_2)=\epsilon(j_2)^{\ga-1}=1$ for all $\ga \in \Ga$.
 	Thus, we obtain $\chi_\ga(j_2)=\pm 1 \,\, \forall \,\, \ga \in \Ga$.
 	Depending on the value, the Dirichlet character $\chi_\gamma$ is unramified or ramified. 
 	For elements $i \in I_T(F_v)$, we have  $\chi_\gamma(i)=1$ as $\alpha(i) \in 
 	F_v^\times$. This is also the case for the elements of $\Gamma_1$, one part 
 	of the wild inertia group, as $r=0,1 (r<s)$ [cf. Lemma~\ref{lem8}].
 \end{proof}
 Note that the above is true for any $\ga_2$-element. Since $f$ is  non-CM, 
 we can and do choose an auxiliary prime $p^\dagger$ stated in the introduction imitating a similar construction of \cite[Section $6.2.3$]{MR3096563}.
 Since $\epsilon^{-1}$ is an extra twist, $\epsilon(p^\dagger)=-1$  
 and we obtain $a_{p^\dagger}^2=-a_{p^\dagger}^2
 \epsilon(p^\dagger)^{-1} \in F^\times$.
 We deduce that  
 $\alpha(j_2) \equiv \alpha(\Frob_{p^\dagger}) 
 \equiv a_{p^\dagger} \,\, \text{mod} \,\, F_v^\times$.  
  Consider the Frobenius element $g_\pi$ in $G_K$ and two fields 
 $F_v^{'}=F_v(b, \zeta_{2^s}+\zeta_{2^s}^{-1})$
 and $F_v^{''}=F_v(b, \zeta_{2^r})$ as in Section~\ref{Statement}.
 \begin{lemma}
  \label{g1}
 	Let $p=2$ be a dihedral supercuspidal prime for $f$ and  $\zeta_{2^s}+ \zeta_{2^s}^{-1} \neq 0$.
	\begin{itemize}
	\item
	If $g \in G_K$ and $\alpha(g) \notin (F_v^{'})^\times$, then 
 	$\alpha(g) \equiv \zeta_{2^r}$ mod $(F_v^{'})^\times$.
	\item 
	If $\alpha(g) \notin (F_v^{''})^\times$, then we have
 	$\alpha(g) \equiv (\zeta_{2^s}+\zeta_{2^s}^{-1})$ 
 	mod $(F_v^{''})^\times$. 
	 	\end{itemize}
 \end{lemma}
 
 \begin{proof}
 	Recall that every element $g \in G_K$ can be written as $g_\pi^ni$,
 	for some $n \in \Z$ and $i \in I_K$. Using the homomorphism
 	$\tilde{\alpha}$ and Lemma~\ref{lem8}, we get the result.
 \end{proof}
 \begin{lemma}
 	\label{g2}
 	Let $p=2$ be a dihedral supercuspidal prime for $f$ and  $\zeta_{2^s}+ \zeta_{2^s}^{-1} = 0$.
 	If $g \in G_K$ and $\alpha(g) \notin (F_v^{''})^\times$, then we have
 	$\alpha(g) \equiv a_{p^\dagger}$ mod $(F_v^{''})^\times$. 
 \end{lemma}
 
 \begin{proof}
 	Since $r \in \{0,1\}$ (as $r<s=2$), we must have $\alpha(i) \in F_v^\times$,
 	for all $i \in I_K$ with $\bar{i} \in <\ga_1>$ [cf. Lemma~\ref{lem8}].
 	Then as in lemma above, we get the result.
 \end{proof}
 \subsection{Error terms for $p=2$}
 \label{errorterms}
 First assume that $K \subseteq F_v$ with $\zeta_{2^s}+\zeta_{2^s}^{-1} \neq 0$.
 Define two functions on $G_v$ by
 \begin{gather} \label{twofunctions2}
 f_1(g)=
 \left\{\begin{array}{lll}
 1 & \quad \text{if} \,\, \alpha(g) \in (F_v^{'})^\times \\              
 \zeta_{2^r} & \quad \text{if} \,\, \alpha(g) \notin (F_v^{'})^\times
 \end{array}\right.
 \quad \text{and}  \quad	f_2(g)=
 \left\{\begin{array}{lll}
 1 & \quad 
 \text{if} \,\, \alpha(g) \in (F_v^{''})^\times \\             
 \zeta_{2^s}+\zeta_{2^s}^{-1} & \quad 
 \text{if} \,\, \alpha(g) \notin (F_v^{''})^\times.
 \end{array}\right.
 \end{gather}
 Recall that we defined two quadratic characters $\psi_1,\psi_2$ on $G_v$ and $t_1, t_2 \in F_v^\times$  in Section~\ref{Statement} . 
 We replace the conditions that define $f_1$ and $f_2$
 by the quadratic characters $\psi_1$ and $\psi_2$ respectively defined in that section.
 
 The functions $f_1$ and $f_2$ will induce $2$-cocycles  $c_{f_1}$ and $c_{f_2}$ in $\Br(F_v)$.
 To compute $\mathrm{inv}_v(c_{f_1})$, consider the non-trivial element  
 $\sigma$ of the Galois group $\Gal(F_v(\sqrt{t_1})|F_v)$. 
 The cocycle table of the $2$-cocycle $c_{f_1}$ is given by
 \begin{center}
 	\begin{tabular}{ |c|c|c|c| } 
 		\hline
 		& $1$ & $\sigma$ \\
 		\hline
 		$1$ & $1$ & $1$ \\
 		\hline
 		$\sigma$ & $1$ & $\zeta_{2^{r-1}}$ \\ 
 		\hline
 	\end{tabular}
 \end{center}
  which gives the symbol $(t_1, \zeta_{2^{r-1}})_v$.  Similarly, the cocycle
 table of $c_{f_2}$ gives the symbol $(t_2,(\zeta_{2^s}+
 \zeta_{2^s}^{-1})^2)_v$.
 Define two integers $n_{1,v}, n_{2,v}$ mod $2$ by
 $(-1)^{n_{1,v}}=(t_1, \zeta_{2^{r-1}})_v$ and
 $(-1)^{n_{2,v}}=(t_2,(\zeta_{2^s}+
 \zeta_{2^s}^{-1})^2)_v$.
 Let us consider an integer $n_v$ mod $2$ defined by
 $(-1)^{n_v}=(-1)^{n_{1,v} + n_{2,v}}=
 (t_1, \zeta_{2^{r-1}})_v \cdot (t_2,(\zeta_{2^s}+
 \zeta_{2^s}^{-1})^2)_v$.
	
 We now assume the case $K \nsubseteq F_v$ 
 with $\zeta_{2^s}+\zeta_{2^s}^{-1} \neq 0$. Consider a non-trivial element
 $\bar{\sigma}_v \in G_v/G_{KF_v}$ for some $\sigma_v \in G_v$. 
 Then any element $g \in G_v$ can be written as: 
 \begin{equation} \label{decom}
 g=\sigma_v^nh \,\, \text{for some} \,\,
 h \in G_{KF_v} \,\, \text{and} \,\, n \in \{0,1\}.
 \end{equation} Note that $n=0$
 when $g \in G_{KF_v}$. Using this decomposition one can extend
 $f_1$ and $f_2$ uniquely to $G_v$, call it $F_1$ and $F_2$,
 by defining $F_1(g)=f_1(h)$ and $F_2(g)=f_2(h)$.
 The inflation map 
$  \Inf :{}_2 \HH^2(G_{KF_v},(\bar{F_v}^\times)^{\Gal(KF_v|F_v)}) \hookrightarrow
 {}_2 {\HH}^2(G_v,\bar{F_v}^\times)$
 sends $c_{f_1}$ and $c_{f_2}$ to $c_{F_1}$ and $c_{F_2}$ respectively and hence
 their classes are same in their respective Brauer groups.
  
 \begin{definition}
 	\label{auxf}
 	(Auxillary functions for $s=2$)
 	Assume that $\zeta_{2^s}+\zeta_{2^s}^{-1}=0$.
 	If $K \subseteq F_v$, we define a function $f$ on $G_v(\subseteq G_K)$ by
 	\begin{eqnarray} \label{eqn2}
 	f(g)=
 	\begin{cases}
 	1, & \quad \text{if} \,\, \psi_2(g)=1,\\
 	a_{p^\dagger}, & \quad \text{if} \,\, \psi_2(g)=-1.
 	\end{cases}
 	\end{eqnarray}
 	When $KF_v|F_v$ is quadratic, we use the above decomposition
 	to extend $f$ uniquely to $G_v$, call it $F$.
 \end{definition}
 As before, both
 $c_f$ and $c_F$ have the same cocycle class in their respective Brauer groups. The cocycle table of $c_f$ is determined by the symbol $(t_2,a_{p^\dagger}^2)_v$ with $t_2$ as above.
 In order to keep same notation, we write $(-1)^{n_v}=(t_1,\zeta_{2^{r-1}})_v
 \cdot (t_2,(\zeta_{2^s}+\zeta_{2^s})^2)_v$, if $s\neq 2$, or $(t_2,a_{p^\dagger}^2)_v$,
 if $s=2$. 
\subsection{The case $K \subseteq F_v$ or $KF_v|F_v$ unramified quadratic}
Let $f_1,f_2,F_1,F_2,f,F$ and $n_v$ be as in the previous section.
We now prove Theorem~\ref{nonexceptional2} with this assumption.

\begin{proof}
	First assume that $\zeta_{2^s}+\zeta_{2^s}^{-1} \neq 0$.
	We now consider a function $\alpha'$ on $G_v$ defined as follows:
	$\alpha'=\frac{\alpha}{f_1f_2}$, if $K \subseteq F_v$ and
	$\alpha'= \frac{\alpha}{F_1F_2}$, if $KF_v|F_v$ unramified quadratic.
	With the assumption as above, we have  $I_v=I_{KF_v} \subseteq I_K$. 
	 We can apply Lemma~\ref{lem8} to determine $\alpha(i) \,\, \forall \,\, i \in I_v$. 
	If  $KF_v|F_v$ is unramified quadratic, we have $I_v=I_{KF_v}$ and hence,
	in the decomposition~(\ref{decom}) for any element of $I_v$,
	we must have $n=0$. As a map on $I_v$, we have $\frac{\alpha}{F_1F_2}=\frac{\alpha}{f_1f_2}$.
	By unravelling the definition of $\alpha'$, we see that 
	the two conditions of Lemma~\ref{lem1}
	are satisfied by $S=\alpha'$ and $t=\epsilon'$. 
	Hence, we obtain 
	\[
	  \mathrm{inv}_v(c_{\alpha'})=
  	  \frac{1}{2} \cdot v\big(\frac{\alpha'^2}{\epsilon'}(\Frob_v) \big)
	  =\frac{1}{2} \cdot v\big(\frac{\alpha^2}{\epsilon'}(\Frob_v) \big)
	  =\frac{1}{2} \cdot f_v \cdot v(a_{p'}^2 \epsilon(p')^{-1}) \,\,
	  \text{mod} \,\, \Z.
	\]
	We conclude that, 
	$
	\mathrm{inv}_v(c_{\alpha})=\frac{1}{2} \cdot 
	\big( f_v \cdot v(a_{p'}^2 \epsilon(p')^{-1})
	+  n_v \big) \,\, \text{mod} \,\, \Z.
	$
	
	We now assume that $\zeta_{2^s}+\zeta_{2^s}^{-1} = 0$. Define a map $\alpha'$ on $G_v$ by: 
	$\alpha'=\frac{\alpha}{f}$, if $K \subseteq F_v$ and
	$\alpha'= \frac{\alpha}{F}$, if $KF_v|F_v$ unramified quadratic.
	In this case, we have $I_v=I_{KF_v} \subseteq I_K$.
	As before since $\frac{\alpha}{F}=\frac{\alpha}{f}$ on $I_v$,
	using Lemma~\ref{g2}
	the two conditions of Lemma~\ref{lem1} are satisfied by
	$S=\alpha'$ and $t=\epsilon'$. Thus, we get the required result
	as before. 
\end{proof}

We now prove Corollary~\ref{maincoro}.
\begin{proof}
 If $N_2=2$, then the extension $K|\Q_2$ is unramified [cf. Section~\ref{galrep}].
By the Lemmas~\ref{l1} and~\ref{lem8},  we have $\alpha(\iota) \in F_v^\times$ for all $\iota \in I_K$. 
 Since $I_v=I_{KF_v} \subseteq I_K$,
 this is true for all $\iota \in I_v$.
  We choose $S=\alpha$ and use Lemma~\ref{lem1} to complete the proof.
\end{proof}

 \subsection{The case $KF_v|F_v$ ramified} 
 We now prove Theorem~\ref{nonexceptional2}  with the present assumption. 
\begin{proof}
	As $KF_v|F_v$ is a ramified quadratic extension, we choose an element $a_0 \in \sO_v^\times$ which is not a norm of the extension. 
	Let $i_0 \in G_v$ be an element which is mapped to 
	$a_0$ under the reciprocity map.
	Assume that $i_0$ is an element of the tame inertia part.
	Since $I_T(F_v)^2=I_T(KF_v)$, we have that $i_0 \in I_T(F_v) 
	\setminus I_T(KF_v)$. Since $i_0^2$ belongs to the tame inertia part
	of $K$, we deduce that
	$\alpha(i_0) \equiv \sqrt{d_0} \pmod {F_v^\times}$
	for some $d_0 \in F_v^\times$ [cf. Lemma~\ref{lem8}].
	
	First assume that $s \neq 2$. Define a function $f$ on $G_v$ by
	\begin{eqnarray} \label{ff}
	   f(g) = 1, \,\, \text{if} \,\, g \in G_{KF_v} \,\, \text{and} \,\,
	   f(g)=\sqrt{d_0}, \,\, \text{if} \,\, g \in G_v \backslash G_{KF_v}
	\end{eqnarray}  
	and define the function $F$ on
	$\Gal(KF_v|F_v)$ which is $1$ for the identity element and $\sqrt{d_0}$ else.
	Consider the inflation map 
	 $ \Inf:\HH^2(F_v(\pi)|F_v) \hookrightarrow \HH^2(\bar{F_v}|F_v)$  as before and it is easy to check  that
	 $\Inf^{-1}(c_f)=c_F$  and the cocycle table of $c_F$ is given by:
	\begin{center}
		\begin{tabular}{ |c|c|c|c| } 
			\hline
			& $1$ & $\sigma$ \\
			\hline
			$1$ & $1$ & $1$ \\
			\hline
			$\sigma$ & $1$ & $d_0$ \\ 
			\hline
		\end{tabular}
	\end{center}
	
	\noindent
	which gives the symbol $(\pi^2, d_0)_v$. 
	We now define an integer $n_v^{'}$ mod $2$
	by $(-1)^{n_v^{'}}=(-1)^{n_v} \cdot (\pi^2, d_0)_v$.
	Let $F_1,F_2$ be the functions 
	as in the previous theorem. 
	We apply Lemma \ref{lem1} to compute the invariant of the cocycle $[c_S]$ with 
	$
	S(g)=\frac{\alpha(g)}{F_1(g)F_2(g)f(g)}\,\, \text{on} \,\, G_v
	$
	and $t=\epsilon'$.

	We now deduce that $S(i) \in F_v^{\times}$,  
	for all $i \in I_v$.  Consider the element $i_0 \in G_v \setminus G_{KF_v}$ and  hence $\bar{i_0}$ is a
	non-trivial element in $G_v/G_{KF_v}$. We will consider 
	the decomposition~(\ref{decom}) with respect to 
	the element $i_0$ instead of $\sigma_v$. 
	For $i \in I_v \setminus I_{KF_v}$, we have
	$i=i_0 \tilde{i}$, 
	for some $\tilde{i} \in I_{KF_v} (\subseteq I_K)$. We obtain
	\[
	  S(i) 
	  = 
	  \frac{\alpha(i)}{F_1(i) F_2(i) f(i)}  
	  \equiv 
	  \frac{\alpha(i_0) \alpha(\tilde{i})}{F_1(i) F_2(i) f(i)} 
	  \equiv 
	  \frac{\alpha(\tilde{i})}{F_1(i) F_2(i)} 
	  \equiv 
	  \frac{\alpha(\tilde{i})}{f_1(\tilde{i}) f_2(\tilde{i})} 
	  \equiv 1 \,\, \text{mod} \,\, F_v^\times.
	\]
	The last equality follows from the fact that $\alpha(i) \equiv f_1(i) f_2(i) \pmod {F_v^\times}$ 
	for all $i \in  I_{KF_v}$. 
	As usual,  the cocycle $c_S$ can be decomposed as
	$c_S=c_\alpha c_{F_1} c_{F_2} c_f$. 
	By Lemma~\ref{lem1}, we deduce that
	$
	\mathrm{inv}_v(c_S)=
	\frac{1}{2} \cdot v\big(\frac{S^2}{\epsilon'}(\Frob_v) \big)
	=\frac{1}{2} \cdot v\big(\frac{\alpha^2}{\epsilon'}(\Frob_v) \big)
	=\frac{1}{2} \cdot f_v \cdot v(a_{p'}^2 \epsilon(p')^{-1}) \,\,
	\text{mod} \,\, \Z.
	$
	Hence, we conclude that:
    \[
      \mathrm{inv}_v(c_\alpha) = \mathrm{inv}_v(c_S) + 
      \mathrm{inv}_v(c_{f_1}) + \mathrm{inv}_v(c_{f_2}) +
      \mathrm{inv}_v(c_f)  \\
      = \frac{1}{2} \cdot 
      \big( f_v \cdot v(a_{p'}^2 \epsilon(p')^{-1}) + 
      n_v' \big) \,\, \text{mod} \,\, \Z.
    \]
	We now assume that $s=2$. 
	Let $\alpha'(g)=\frac{\alpha(g)}{F(g)f(g)}$ on $G_v$, where
	$F$ be the function as in Definition~\ref{auxf}
	and the function $f$ be as in (\ref{ff}). 
	Then the cocycle $c_\alpha$ can be decomposed as
	$c_\alpha=c_{\alpha'}c_fc_F$.
	Define an integer $n_v^{''}$ mod $2$ by
	$(-1)^{n_v^{''}}=(t_2,a_{p^\dagger}^2)_v \cdot (\pi^2,d_0)_v$.
	Two conditions of Lemma~\ref{lem1} are satisfied by
	$S=\alpha'$ and $t=\epsilon'$ as before.
	Hence, we have
	$
	\mathrm{inv}_v(c_{\alpha'})= \frac{1}{2} \cdot f_v \cdot v(a_{p'}^2 \epsilon(p')^{-1}) \,\,
	\text{mod} \,\, \Z
	$. 
	Thus,
	\[
	  \mathrm{inv}_v(c_\alpha) = \mathrm{inv}_v(c_{\alpha'})
	  + \mathrm{inv}_v(c_f) + \mathrm{inv}_v(c_F) \\
	  = \frac{1}{2} \cdot \big( f_v \cdot v(a_{p'}^2 \epsilon(p')^{-1})
	  +n_v^{''} \big) \,\, \text{mod} \,\, \Z.
	\]
\end{proof}

 The next lemma will determine the value of $d_0$
 in some special cases.
\begin{lemma}
	Assume  $F=\Q$ and the supercuspidal prime $p$ satisfies the second condition of Theorem~\ref{nonexceptional2}.
	The quantity $d_0$ involved in the error term
	is equal to $a_{p^{'''}}^2$ except $K=\Q_2(\sqrt{2})$ and $\Q_2(\sqrt{-6})$. 
\end{lemma}	

\begin{proof}
	When $F=\Q$, the ramified quadratic extension becomes 
	$K|\Q_2$.	
	We have that $a_0=-1$ is not a norm of the extension $K|\Q_2$ 
	except $K=\Q_2(\sqrt{d})$ with $d=2,-6$, see
	\cite[p. $34$]{MR2665139}.
	In this case,
	we find out the value of $d_0$.
	For a prime $p^{'''}$ chosen before,
	$
	\alpha(i_0)^{\gamma-1} = \chi_\gamma(i_0)  
	\overset{(\ref{idelic})}{=} \chi_{\gamma,2}(-1)^{-1}    
	= \chi_{\gamma}(p^{'''})
	= \alpha(\Frob_{p^{'''}})^{\gamma-1},
	$
	for all $\ga \in \Ga$ and so 
	$\alpha(i_0) \equiv a_{p^{'''}}  \,\, \text{mod} 
	\,\, F_v^\times$.
\end{proof}

 \section{Ramifications for primes lying above non-dihedral supercuspidal prime  $p=2$}
 Let $\rho_2(f)$
 be the local representation of the Weil-Deligne group of $\Q_2$ 
 associated to $f$ at the prime $p=2$ and denote by $\tilde{\rho_2}$ the projective 
 image of $\rho_2:=\rho_2(f)$.  When the inertia group $I_2$ acts irreducibly, 
 the image of $\tilde{\rho_2}$ is one of three \emph{exceptional} groups
 $A_4,S_4,A_5$. The $A_5$ case
 cannot occur since the Galois group $G_2:=\Gal(\bar{\Q}_2|\Q_2)$
 is solvable. Weil proved in \cite{MR0379445} that over $\Q$,
 the $A_4$ case also does not occur, 
 so the only exceptional case has image $S_4$. 
 For $D=\mathrm{det}(\rho_2(f))$ and $d=\frac{\alpha^2}{D}$, 
 we have a cocycle class decomposition $[X]=[c_D] \cdot [c_d]$, where 
 the cocycles $c_D$ and $c_d$ are given by  
 $
 c_D(g,h)= \frac{\sqrt{D(g)} \sqrt{D(h)}}{\sqrt{D(gh)}}
 , \quad c_d(g,h)= \frac{\sqrt{d(g)} \sqrt{d(h)}}{\sqrt{d(gh)}}$
 respectively. In this section, we find the local Brauer class $[X_v]=[c_D]_v \cdot [c_d]_v$ for any $v \mid 2$. Let $D_{K'}$ be the discriminant of the field $K'$ cut out by $\ker(d)$.
 
The following Lemma is a straightforward adaptation in our setting of 
Lemma $7.3.17$ of \cite{Banerjee:Thesis:2010}. 
 \begin{lemma} \label{D-1}
 	The $2$-cocycle $[c_D]_v=1$ if and only if $D(-1)=1$.
 \end{lemma}

 Recall that  $f^{\gamma} \equiv f \otimes \chi_{\gamma}$ implies that
 $\rho_2(f)^{\gamma} \sim \rho_2(f)  \otimes \chi_{\gamma}$
 and taking determinant gives $\mathrm{det}(\rho_2(f))^{\gamma}=
 \mathrm{det}(\rho_2(f))(\chi_{\gamma})^2$. 
 Thus, we obtain $d(g)^{\gamma-1}=1$ and so $d(g) \in F_v^\times$. 
 Consider the topological space  $F_v^\times/F_v^{\times 2}$ with  the discrete topology. 
 We deduce that  $ d: G_v \rightarrow F_v^\times/F_v^{\times 2}$
 is a continuous homomorphism and hence
 $G_v/ \text{ker}(d) \cong \Gal(K'|F_v) \cong (\Z/2\Z)^m$
 for some elementary $2$-extension $K'$ of $F_v$.
 For each $i=1, \cdots ,m$, let $\sigma_i \in \Gal(K'|F_v)$ be the element corresponding to 
 $(0, \cdots,1, \cdots, 0) \in (\Z/2\Z)^m$ ($1$ in the $i$-th position). Define $t_j \in F_v^\times (1 \leq j \leq m)$
 as follows:
 \[
 \sigma_i(\sqrt{t_j})=\delta_{ij} \sqrt{t_j}.
 \]
 Lift $\sigma_i$ to an element of $G_v$ denoted also by $\sigma_i$
 and set $d_i:=d(\sigma_i) \in F_v^\times/F_v^{\times 2}$. 
 As in \cite{MR1650241}, the class of $c_d$ is given by
$[c_d]_v=(t_1, d_1)_v \cdots (t_m,d_m)_v$.

 \begin{prop} \label{DK}
    \[
    [c_d]_v=(2, D_{K'})_v.
    \]
 \end{prop}
\begin{proof}
We first prove that
  $\text{ker}(\tilde{\rho_2})  \subset \text{ker}(d)$. 
Suppose $g \in  \text{ker}(\tilde{\rho_2}) $, then $\rho_2(g)$ is a scalar matrix 
$\rho_2(g)=\left(\begin{smallmatrix}
a & 0\\
0 & a\\
\end{smallmatrix}\right)$.
Thus, the quantity $\mathrm{Tr}(\rho_2(g)^2)/\mathrm{det}(\rho_2(g))$ is $4$. As trace 
is non-zero, using the part $(2)$ of Proposition~\ref{prop1}, 
this quantity is equal to $d(g)$ up to an element of $F^{\times 2}$
and so $d(g) \equiv 1$ mod $F_v^{\times 2}$.
Hence, we get the inclusion.

Thus, there is an onto map
$S_4=G_v \slash  \text{ker}(\tilde{\rho_2}) \rightarrow G_v \slash \text{ker}(d)$. 	
Only $2$-subgroup that can be quotient of $S_4$ is the trivial group or $\Z/2\Z$. 
Thus $m$ is either $0$ or $1$. 
Since every element of the projective image of $\rho_2$ (which is $S_4$)
has order $1,2,3$ or $4$, we have 
$\mathrm{Tr}(\rho_2(g)^2)/\mathrm{det}(\rho_2(g)) \in \{4,0,1,2\}$ for every $g \in G_v$.
For all $g \in G_v$, observe that $d(g) \in \{4,0,1,2\}$  up to 
$F_v^{\times 2}$ \cite[p. $264$]{MR0419358}. 
We conclude that
either $d(g) \in F_v^{\times 2}$ or $d(g) \equiv \sqrt{2}$ mod $F_v^\times$,
for all $g \in G_v$. The value of $m \in \{0,1\}$ depends on this.
Since the projective image of $\rho_2$ is $S_4$, there is an element
$g \in G_v$ whose projective image in $S_4$ is a $4$-cycle.
For such an element $g$, we have $d(g) \equiv 2$ mod $F_v^{\times 2}$
and so we conclude that $m=1$. Thus, the field $K'$ cut out by the 
kernel of the homomorphism $d$ must be a quadratic field. 
 As
$t_1=D_{K'}$ (the discriminant of $K'$) and $d_1=2$ up to
an element of $F_v^{\times 2}$,  the class $[c_d]_v$ is determined by the symbol $(2,D_{K'})_v$. We obtain the result. 
\end{proof}

 We now prove Theorem~\ref{exceptional} and Corollary~\ref{Depsilon}.
\begin{proof}
  Since $[X_v] \sim [c_d]_v \cdot [c_D]_v$, we obtain the Theorem~\ref{exceptional} 
  for a non-dihedral prime $p=2$. 
 
 By the isomorphism $(*)$ [cf. Section~\ref{galrep}] we can replace
 $D$ by det$(\rho_{f,2})=\chi_2^{k-1} \epsilon$, where $\chi_2$
 is the $2$-adic cyclotomic character. When $k$ is odd, the cocycle
 class of $c_D$ is same as the cocycle class of $c_\epsilon$, that is,
 $c_D \sim c_\epsilon$, where the $2$-cocycle $c_\epsilon$ is defined
 as follows:
 $c_\epsilon(g,h)= \frac{\sqrt{\epsilon(g)} \sqrt{\epsilon(h)}}{\sqrt{\epsilon(gh)}}$.
Apply Lemma~\ref{D-1} and observe that $[c_D]=[c_\epsilon]$ for odd $k$, we obtain  Corollary~\ref{Depsilon}.
 \end{proof}

 \section{Numerical Examples}
 For an odd prime $p$, our results are concurrent with the theorems proved in \cite{MR3391026}. 
 However, the example $(5)$  of loc. cit. shows that $X_v$ is not determined by $m_v$ 
 if $p$ is an unramified ``bad" level zero supercuspidal prime. This example corroborates 
 our Theorem~\ref{extrathm}.
 
 To support our results, we give numerical examples 
 about local ramifications at supercuspidal prime $p=2$. 
 The examples are provided in the table of
 \cite{MR2146605}. Using {\texttt Sage} and $L$-function and modular forms database ({\texttt LMFDB}), we determine the $v$-adic valuation of the trace
 of adjoint lift at the prime $p'$. 
 \begin{enumerate}
 	\item 
 	     $f \in S_3(20,[0,3])$ with $E=\Q(\sqrt{-1})$ and $F=\Q$. Since $N_2=2$ 
 	     the prime $p=2$ is an unramified dihedral supercuspidal prime. 
 	     We choose $p'=17$. Using {\texttt Sage} we check that $a_{p'}=1-i$ and
 	     hence $v_2(a_{p'}^2\epsilon(p')^{-1})=1$,
 	     so $X_v$ is ramified.
 	\item
 	     $f \in S_5(36,[0,3]), E=\Q(\sqrt{-2})$ and $F=\Q$. Here $p=2$ is an
 	     unramified dihedral supercuspidal prime as $N_2=2$. 
 	     We choose $p'=29$ and compute that $a_{p'}^2=a_{29}^2=-421362=-2 \cdot 459^2$.
 	     Hence $v_2(a_{29}^2\epsilon(29)^{-1})=1$, so
 	     $X_v$ is ramified.
 	\item
 	     $f \in S_3(24,[0,0,1]), E=\Q(\sqrt{-2}), F=\Q$. Here $p=2$ is a
 	     non-dihedral supercuspidal prime for $f$ \cite[Section $6$]{MR2223979}.
 	     We have $D_K=64$ \cite[Section $6$]{MR2223979}
 	     and $\epsilon(-1)=-1$.
 	     Hence $X_v$ is ramified.
 \end{enumerate}

	\bibliographystyle{crelle}
	\bibliography{Eisensteinquestion.bib}

\def\cprime{$'$}
\begin{thebibliography}{10}

\bibitem{Banerjee:Thesis:2010}
D.~Banerjee, {Endomorphism algebras of modular motives}, Ph.D. thesis, Tata
  Institute of fundamental Research, Mumbai, India (2010).

\bibitem{MR2770587}
D.~Banerjee and E.~Ghate, \emph{Crossed product algebras attached to weight one
  forms}, Math. Res. Lett. \textbf{18} (2011), no.~1,  139--149.

\bibitem{MR3096563}
---{}---{}---, \emph{Adjoint lifts and modular endomorphism algebras}, Israel
  J. Math. \textbf{195} (2013), no.~2,  507--543.

\bibitem{MR3391026}
S.~Bhattacharya and E.~Ghate, \emph{Supercuspidal ramification of modular
  endomorphism algebras}, Proc. Amer. Math. Soc. \textbf{143} (2015), no.~11,
  4669--4684.

\bibitem{MR2038777}
A.~F. Brown and E.~P. Ghate, \emph{Endomorphism algebras of motives attached to
  elliptic modular forms}, Ann. Inst. Fourier (Grenoble) \textbf{53} (2003),
  no.~6,  1615--1676.

\bibitem{MR2234120}
C.~J. Bushnell and G.~Henniart, The local {L}anglands conjecture for {$\rm
  GL(2)$}, Vol. 335 of \emph{Grundlehren der Mathematischen Wissenschaften
  [Fundamental Principles of Mathematical Sciences]}, Springer-Verlag, Berlin
  (2006), ISBN 978-3-540-31486-8; 3-540-31486-5.

\bibitem{MR870690}
H.~Carayol, \emph{Sur les repr\'esentations {$l$}-adiques associ\'ees aux
  formes modulaires de {H}ilbert}, Ann. Sci. \'Ecole Norm. Sup. (4) \textbf{19}
  (1986), no.~3,  409--468.

\bibitem{CasselFrohlich}
J.~Casselss and A.~Fr\"{o}hrlich, Algebraic number theory, 2nd Edition, London
  Mathematical Society (2010).

\bibitem{MR1779803}
P.~Colmez and J.-M. Fontaine, \emph{Construction des repr\'esentations
  {$p$}-adiques semi-stables}, Invent. Math. \textbf{140} (2000), no.~1,
  1--43.

\bibitem{MR0347738}
P.~Deligne, \emph{Formes modulaires et repr\'esentations de {${\rm GL}(2)$}}
  (1973) 55--105. Lecture Notes in Math., Vol. 349.

\bibitem{MR1218392}
I.~B. Fesenko and S.~V. Vostokov, Local fields and their extensions, Vol. 121
  of \emph{Translations of Mathematical Monographs}, American Mathematical
  Society, Providence, RI (1993).

\bibitem{MR2146605}
E.~Ghate, E.~Gonz{\'a}lez-Jim{\'e}nez, and J.~Quer, \emph{On the {B}rauer class
  of modular endomorphism algebras}, Int. Math. Res. Not.  (2005), no.~12,
  701--723.

\bibitem{MR2860430}
E.~Ghate and N.~Kumar, \emph{{$(p,p)$}-{G}alois representations attached to
  automorphic forms on {${\rm GL}_n$}}, Pacific J. Math. \textbf{252} (2011),
  no.~2,  379--406.

\bibitem{MR2471916}
E.~Ghate and A.~M{\'e}zard, \emph{Filtered modules with coefficients}, Trans.
  Amer. Math. Soc. \textbf{361} (2009), no.~5,  2243--2261.

\bibitem{MR863740}
K.~Iwasawa, Local class field theory, Oxford Science Publications, The
  Clarendon Press, Oxford University Press, New York (1986), ISBN
  0-19-504030-9. Oxford Mathematical Monographs.

\bibitem{MR617867}
F.~Momose, \emph{On the {$l$}-adic representations attached to modular forms},
  J. Fac. Sci. Univ. Tokyo Sect. IA Math. \textbf{28} (1981), no.~1,  89--109.

\bibitem{MR2962318}
J.~Nekovar, \emph{Level raising and anticyclotomic {S}elmer groups for
  {H}ilbert modular forms of weight two}, Canad. J. Math. \textbf{64} (2012),
  no.~3,  588--668.

\bibitem{MR1650241}
J.~Quer, \emph{La classe de {B}rauer de l'alg\`ebre d'endomorphismes d'une
  vari\'et\'e ab\'elienne modulaire}, C. R. Acad. Sci. Paris S\'er. I Math.
  \textbf{327} (1998), no.~3,  227--230.

\bibitem{MR0419358}
K.~A. Ribet, \emph{On {$l$}-adic representations attached to modular forms},
  Invent. Math. \textbf{28} (1975) 245--275.

\bibitem{MR594532}
---{}---{}---, \emph{Twists of modular forms and endomorphisms of abelian
  varieties}, Math. Ann. \textbf{253} (1980), no.~1,  43--62.

\bibitem{MR633903}
---{}---{}---, \emph{Endomorphism algebras of abelian varieties attached to
  newforms of weight {$2$}}, in Seminar on {N}umber {T}heory, {P}aris 1979--80,
  Vol.~12 of \emph{Progr. Math.}, 263--276, Birkh\"auser Boston, Mass. (1981).

\bibitem{MR2058653}
---{}---{}---, \emph{Abelian varieties over {$\bf Q$} and modular forms}, in
  Modular curves and abelian varieties, Vol. 224 of \emph{Progr. Math.},
  241--261, Birkh\"auser, Basel (2004).

\bibitem{MR2223979}
A.~Rio, \emph{Dyadic exercises for octahedral extensions. {II}}, J. Number
  Theory \textbf{118} (2006), no.~2,  172--188.

\bibitem{MR1465337}
T.~Saito, \emph{Modular forms and {$p$}-adic {H}odge theory}, Invent. Math.
  \textbf{129} (1997), no.~3,  607--620.

\bibitem{MR1047142}
A.~J. Scholl, \emph{Motives for modular forms}, Invent. Math. \textbf{100}
  (1990), no.~2,  419--430.

\bibitem{MR554237}
J.-P. Serre, Local fields, Vol.~67 of \emph{Graduate Texts in Mathematics},
  Springer-Verlag, New York (1979).

\bibitem{MR0314766}
G.~Shimura, Introduction to the arithmetic theory of automorphic functions,
  Publications of the Mathematical Society of Japan, No. 11. Iwanami Shoten,
  Publishers, Tokyo (1971). Kan{\^o} Memorial Lectures, No. 1.

\bibitem{MR2665139}
---{}---{}---, Arithmetic of quadratic forms, Springer Monographs in
  Mathematics, Springer, New York (2010), ISBN 978-1-4419-1731-7.

\bibitem{MR0379445}
A.~Weil, \emph{Exercices dyadiques}, Invent. Math. \textbf{27} (1974) 1--22.

\end{thebibliography}

\end{document}